\documentclass[11pt, reqno, a4paper]{amsart}
\usepackage{a4,amsmath,amssymb,amscd,verbatim,bbm,graphicx,enumerate, blkarray}

\usepackage{amsfonts,amsthm}
\usepackage[colorlinks,citecolor=cyan,linkcolor=magenta]{hyperref}

\usepackage{graphicx}
\usepackage{bmpsize}
\usepackage{MnSymbol}
\usepackage{color}
\usepackage{mathrsfs}
\usepackage{tikz-cd}

\usepackage{enumitem}

\usepackage{cleveref}

\newtheorem{theorem}{Theorem}
\newtheorem{proposition}[theorem]{Proposition}
\newtheorem{corollary}[theorem]{Corollary}
\newtheorem{lemma}[theorem]{Lemma}

\theoremstyle{definition}

\newtheorem{definition}[theorem]{Definition}

\newtheorem{remark}[theorem]{Remark}

\newtheorem{question}[theorem]{Question}

\numberwithin{equation}{section}
\numberwithin{theorem}{section}

\newcommand{\bfv}{\mathbf{v}}
\newcommand{\vv}{\mathbf{v}}

\newcommand{\vw}{\mathbf{w}}
\newcommand{\vu}{\mathbf{u}}
\newcommand{\vx}{\mathbf{x}}
\newcommand{\vy}{\mathbf{y}}

\renewcommand{\b}{\beta}
\newcommand{\g}{\gamma}
\newcommand{\G}{\Gamma}
\renewcommand{\d}{\delta}
\newcommand{\D}{\mathbf{D}}
\newcommand{\bfL}{\mathbf{\Lambda}}

\renewcommand{\k}{\kappa}

\renewcommand{\l}{\lambda}
\renewcommand{\L}{\Lambda}

\renewcommand{\r}{\rho}

\newcommand{\F}{\Phi}
\newcommand{\x}{\xi}

\newcommand{\y}{\eta}

\newcommand{\p}{\psi}

\DeclareMathOperator{\Int}{Int}

\newcommand{\Cc}{{\mathcal C}}
\newcommand{\Dc}{{\mathcal D}}

\newcommand{\Fc}{{\mathcal F}}

\newcommand{\C}{{\mathbb C}}

\newcommand{\N}{{\mathbb N}}
\newcommand{\R}{{\mathbb R}}

\newcommand{\Z}{{\mathbb Z}}

\newcommand{\calB}{\mathcal{B}}
\newcommand{\calC}{\mathcal{C}}
\newcommand{\calD}{\mathcal{D}}

\newcommand{\calF}{\mathcal{F}}
\newcommand{\calG}{\mathcal{G}}
\newcommand{\calH}{\mathcal{H}}

\newcommand{\J}{\mathcal{J}}

\newcommand{\calM}{\mathcal{M}}

\newcommand{\calT}{\mathcal{T}}

\newcommand{\calW}{\mathcal{W}}
\newcommand{\calX}{\mathcal{X}}

\newcommand{\scrC}{\mathscr{C}}
\newcommand{\scrD}{\mathscr{D}}

\newcommand{\scrH}{\mathscr{H}}

\newcommand{\scrM}{\mathscr{M}}

\newcommand{\scrV}{\mathscr{V}}

\newcommand{\bbP}{\mathbb{P}}
\newcommand{\bbS}{\mathbb{S}}

\newcommand{\al}{\alpha}
\newcommand{\gam}{\gamma}

\newcommand{\del}{\delta}

\newcommand{\Del}{\Delta}
\newcommand{\ep}{\epsilon}
\newcommand{\thet}{\theta}

\newcommand{\lam}{\lambda}
\newcommand{\Lam}{\Lambda}
\newcommand{\sig}{\sigma}

\newcommand{\om}{\omega}

\newcommand\GL{\operatorname{GL}}

\newcommand\PSL{\operatorname{PSL}}
\newcommand\SL{\operatorname{SL}}

\newcommand{\bs}{\backslash}

\newcommand{\ra}{\rightarrow}

\def\({\left(}
\def\){\right)}

\def\l\{ { \left\{ }
\def\r\}{\right\}}
\def\wt{\widetilde}
\def\wh{\widehat}

\def\wbar{\overline}
\def\ov{\overline}
\def\id{{\rm id}}
\newcommand{\tred}[1]{\textcolor{red}{#1}}

\def\Cay{{\rm Cay}}
\def\Leb{{\rm Leb}}
\def\ev{{\rm ev\,}}
\def\1{{\bf 1}}

\def\gr{{\rm gr}}
\def\conj{{\bf conj}}

\def\ac{{\rm ac}}
\def\sing{{\rm sing}}

\def\H{{\mathbb H}}

\def\bS{{\bf S}}

\def\CAT{{\rm CAT}}

\def\Dil{{\rm Dil}}

\DeclareMathOperator{\diam}{diam}

\newcommand{\violet}[1]{{\color{violet}#1}}

\newcommand\numberthis{\addtocounter{equation}{1}\tag{\theequation}}

\date{\today}

\title{The joint translation spectrum and Manhattan manifolds}
\author{\small{Stephen Cantrell, Eduardo Reyes and Cagri Sert}}

\begin{document}
\maketitle
\begin{abstract}
We define and study geometric versions of the Benoist limit cone and matrix joint spectrum, which we call the translation cone and the joint translation spectrum, respectively. These new notions allow us to generalize the study of embeddings into products of rank-one simple Lie groups and to compare group actions on different metric spaces, quasi-morphisms, Anosov representations and many other natural objects of study.

We identify the joint translation spectrum with the image of the gradient function of a corresponding Manhattan manifold: a higher dimensional version of the well known and studied Manhattan curve. As a consequence we deduce many properties of the spectrum. For example we show that it is given by the closure of the set of all possible drift vectors associated to finitely supported, symmetric, admissible random walks on the associated group. 
\end{abstract}
\maketitle

\section{Introduction}

Let $S$ be a compact subset of matrices in $\GL_n(\C)$. The \textit{joint spectral radius} of $S$ is the limit
\[
R(S) = \lim_{m\to\infty} \sup_{A \in S^m} \|A\|^{1/m}
\]
where $S^m = \{s_1\cdots s_m : s_1,\ldots,s_m\in S \}$ is the set of $m$-fold products of matrices in $S$. The limit exists by submultiplicativity and does not depend on the choice of norm. It is a  conjugation invariant: $R(gSg^{-1})=R(S)$ for all $g \in \GL_n(\C)$. This notion was introduced by Rota and Strang \cite{rota-strang} in the 60's and has since been extensively studied in a variety of contexts, pure and applied, in particular in the study of wavelets, control theory, ergodic optimization and beyond; see for example \cite{daubechies-lagarias, lagarias-wang, berger-wang}, \cite{barabanov, gurvits}, \cite{bousch-mairesse, jenkinson.survey, morris.mather} respectively, and the many references therein.

In \cite{breuillard-sert}, Breuillard--Sert have introduced the notion of \textit{joint spectrum}, a multidimensional version of the the joint spectral radius that encapsulates possible growth rates of all of the singular values for products of matrices in $S$. More precisely, given a matrix $g \in \GL_n(\C)$, let $\sigma_1(g) \geq \ldots  \geq \sigma_n(g)>0$ be its singular values and let $\kappa: \GL_n(\C) \to \R^n$ be the Cartan projection: 
\begin{equation}\label{eq.def.cartan}
\kappa(g):=(\log \sigma_1(g),\ldots,\log \sigma_n(g)).  
\end{equation}
Under an irreducibility assumption, the authors showed that the sequence of subsets
\[
\frac{1}{m}\kappa(S^m) := \left\{ \frac{\kappa(A)}{m} : A \in S^m \right\} \subset \R^n \ \text{for $m\ge 1$}
\]
converges in the Hausdorff metric to a compact subset $\J(S)$ of $\R^n$. This limiting set is called the joint spectrum and after introducing $\J(S)$, in \cite{breuillard-sert} the authors prove various interesting results regarding its properties. For example they show that $\J(S)$ is a convex body (a compact convex set with non-empty interior) and that any convex body can be realized as the joint spectrum of some compact set of matrices satisfying the Zariski-density asssumption. They showed that it is closely related to certain limit theorems in the theory of random matrix products. In that setting, the joint spectrum encodes an asymptotic radial information (of powers of a generating set) in addition to the \textit{Benoist limit cone}, introduced much earlier by Benoist \cite{benoist.cone}, which corresponds to the projective image of the joint spectrum but does not depend on the generating set.

In a geometric direction, the notion of joint spectral radius was adapted to isometric actions on metric spaces by Reyes in \cite{oregonreyes.inequalities} and Breuillard--Fujiwara in \cite{breuillard-fujiwara}. These authors proved a geometric analogue of the Berger--Wang identity \cite{berger-wang} for a variety of situations including isometric actions on Gromov-hyperbolic spaces and (higher-rank) symmetric spaces of non-compact type. In the latter work \cite{breuillard-fujiwara}, the authors also establish a geometric analogue of Bochi inequality \cite{bochi} for isometric actions on Gromov-hyperbolic spaces.

The aim of this article is to introduce and study  geometric analogues of the Benoist limit cone \cite{benoist.cone} and joint spectrum  \cite{breuillard-sert}, that we will call, respectively, $\textit{the translation cone}$ and \textit{the joint translation spectrum}, which allow us to simultaneously compare several isometric actions on metric spaces, quasi-morphisms and other natural real-valued potentials. In this sense, our work can be seen as a multi-dimensional version of the aforementioned work by Reyes \cite{oregonreyes.inequalities} and Breuillard--Fujiwara \cite{breuillard-fujiwara}. After establishing its various characterizations and basic properties, we will study the joint translation spectrum from a variety of angles including geodesic currents, random walks on Gromov-hyperbolic groups, and ergodic optimization. We will establish results reminiscent of the simplicity of Lyapunov exponents (due to Guivarc'h--Raugi \cite{guivarch-raugi}). We will obtain a rigidity statement forcing the joint translation spectrum to be a polygon when the underlying metrics are induced by finite generating sets, which we show to not always be the case when we consider more general metrics. Finally, the analysis of the joint translation spectrum turns out to be tightly related to a dual object, that we call \textit{Manhattan manifold} and we investigate this duality. When specialized to comparing two metrics, Manhattan manifold boils down precisely to the well-known Manhattan curve that was introduced by Burger \cite{burger}.

Before proceeding to state precise definitions and results, let us mention that in this work we will focus our attention to Gromov-hyperbolic groups and their actions. 
Although it is possible to establish the notions of  translation cone and joint translation spectrum in more general settings (e.g.\ a larger class of groups and their isometric actions), we choose to work in this setting as we can draw upon techniques from geometry and ergodic theory $\&$ dynamical systems. These allow us to prove sometimes stronger results than the statements for matrix groups (in, for example \cite{breuillard-sert}).

\subsection{Hyperbolic metric potentials} \label{sec.jts}
In this part we introduce the class of (quasi)- metrics and more general class of metric-like functions that we call hyperbolic metric potentials. For hyperbolic groups, they will play a role somewhat reminiscent of log-singular values (or more generally, coordinates of a Cartan subspace) for reductive Lie groups.

Let $\Gamma$ be a non-elementary hyperbolic group with identity element $o \in \Gamma$ and let $\Dc_\G$ denote the collection of left-invariant hyperbolic pseudo-metrics on $\G$ that are quasi-isometric to a word metric associated to a finite symmetric generating set. We are interested in understanding the following class of functions on $\G$.
\begin{definition}
Let $\calH_\G$ be the real vector space of all functions $\psi:\G \times \G \ra \R$ satisfying the following:
    \begin{enumerate}
        \item $\G$-invariance: $\psi(sx,sy)=\psi(x,y)$ for all $x,y,s\in \G$; and,
        \item for any $d_0\in \calD_\G$ there exists $\lam>0$ such that for any $x,y, w\in \G$ we have 
        \[
        |(x|y)_{w}^\psi| \leq \lam (x|y)_w^{d_0}+\lam,
        \]
        where 
        \[
        (x|y)_{w}^\psi:=\frac{(\psi(x,w)+\psi(w,y)-\psi(x,y))}{2}
        \]
        denotes the \emph{Gromov product} of $x,y$ based at $w$ with respect to $\psi$.
    \end{enumerate}   
We will refer to the elements of $\calH_\G$ as \textit{hyperbolic metric potentials}. 
We also let $\calH_\G^{++}\subset \calH_\G$ be the subset of all those $\psi \in \calH_\G$ that are bounded below and proper. That is, those $\psi$ satisfying $\psi(x,y)>-C$ for all $x, y$ in $\G$ for some $C>0$ independent of $x,y$ and such that $\{x\in \G : \psi(o,x)<T\}$ is finite for all $T>0$. It is easily checked to be a blunt convex cone in $\calH_\G$ containing $\calD_\G$.
\end{definition}



Condition (2) above can be understood as quasi-additivity of $\p$ along quasi-geodesics in $\G$, but we note that hyperbolic metric potentials are not necessarily symmetric. This is the case of word metrics for finite asymmetric generating subsets of $\G$. Furthermore, hyperbolic metric potentials can be unbounded from above and below, with main examples being quasimorphisms. Namely, if $\varphi:\G \ra \R$ is a quasimorphism, then $\psi(x,y):=\varphi(x^{-1}y)$ defines a hyperbolic metric potential. See \Cref{sec.H_G} for more examples and a discussion on $\calH_\G$.

\subsection{Translation cone and joint translation spectrum}\label{sec.intro.TC+JTS}
Although the class $\calH_\G$ is much larger than the collection of metrics on $\Gamma$, it is designed to behave as a collection of possibly signed and asymmetric hyperbolic metrics. One manifestation of this feature is that for $\p \in \calH_\G$, the following limit
\[
\ell_\p(x) := \lim_{m\to\infty} \frac{\p(o,x^m)}{m} \quad \text{for} \; x \in \Gamma,
\]
exists (see \Cref{lem:MLSnonegative}) allowing us to define the \emph{stable translation length} functional $\ell_\p:\G \ra \R$. Keeping with the analogy of thinking of functionals in $\calH_\G$ as log-singular values of matrices, the stable translation length functionals $\ell_\p$ with $\p \in \calH_\G$ can be thought of as log-moduli of eigenvalues of matrices, or equivalently coordinates of the Jordan projection (the map given by replacing the singular values in the Cartan projection \eqref{eq.def.cartan} by the moduli of generalized eigenvalues).

Now given finitely many hyperbolic metric potentials $\p_1, \ldots, \p_n \in \calH_\G$ define the functions $\D, \bfL : \G \to \R^n$ by
\begin{align*}
\D(x) = (\p_1(o,x), \ldots, \p_n(o,x))  \quad \text{ and} \quad
\bfL(x) = (\ell_{\p_1}(x), \ldots, \ell_{\p_n}(x)).
\end{align*}
By definitions, for every $x \in \Gamma$, we have $\frac{1}{m}\D(x^m) \to \bfL(x)$ as $m \to \infty$ and $\bfL(x^m)=m\bfL(x)$ for every $m \in \N$. By abuse of notation, we also let $\D$ denote the $n$-tuple $(\psi_1,\ldots,\psi_n)$.

\begin{definition}[Translation cone] 
    We define the \emph{translation cone} of the tuple $\D=(\psi_1,\ldots,\psi_n)$ to be the closed cone in $\R^n$ generated by $\{\bfL(x) :x \in \G\}$. We will denote it by $\mathcal{TC}(\D)$.
\end{definition}
Thinking of the functions $\D$ and $\bfL$ as geometric versions of the Cartan and Jordan projection in linear algebraic setting, the translation cone is then the object analogous to the Benoist limit cone \cite{benoist.cone} in our setting. Indeed, when specialized to the case hyperbolic group $\Gamma$ in a product $\prod_{i=1}^n G_i$ of rank-one simple Lie groups, the translation cone of the tuple $\D$ consisting of $G_i$-invariant metrics on the associated symmetric spaces boils down to the Benoist limit cone. As for the usual Benoist limit cone, we will record in Theorem \ref{thm.intro.cone} below that the translation cone coincides with the asymptotic cone of the displacement vectors and that it is convex. Recall that the asymptotic cone of a subset $T$ of $\R^n$ is  the set of $v \in \R^n$ such that there exists a sequence of positive reals $\alpha_k \to 0$ and a sequence $t_n \in T$ such that $\alpha_k t_k \to v$ as $k \to \infty$ (see \cite[\S 4]{benoist.cone}).

We say that $\p_1,\ldots, \p_k \in \calH_\G$ are \emph{independent} if the only constants $c_1,\ldots, c_k \in \R$ for which
    \[
\sup_{x,y \in \G} \left|    \sum_{i=1}^k c_i \p_i(x,y)\right| < \infty
    \]
are $c_1=c_2=\cdots = c_k = 0$. Independence generalizes the notion of two pseudo metrics in $\calD_\G$ not being roughly similar. Given $\D=(\p_1,\ldots,\p_n)$, an obvious obstruction for $\mathcal{TC}(\D)$ to have an empty interior in $\R^n$ is that the tuple $\D$ is not independent. Our first result below shows that the failure of independence is the only reason why $\mathcal{TC}(\D)$ may fail to have a non-empty interior and summarizes the aforementioned properties of $\mathcal{TC}(\D)$.

\begin{theorem}[Translation cone]\label{thm.intro.cone}
Let $\D=(\p_1,\ldots,\p_n)$ be a finite tuple of elements in $\mathcal{H}_\Gamma$. Then, the translation cone $\mathcal{TC}(\D)$ is convex and it coincides with the asymptotic cone of $\D(\Gamma)$. Moreover, it has non-empty interior if and only if $\p_1,\ldots, \p_n$ are independent.    
\end{theorem}

Given a tuple $\D=(\psi_0,\p_1,\ldots,\p_n)$ of  elements in $\calH_\Gamma$, the information of the translation cone in $\R^{n+1}$, thanks to the homogeneity, can be read off on an affine slice, say $x_0=1$ living in $\R^n$. This affine slice gives rise to an object describing asymptotic comparisons of $\p_1,\ldots,\p_n$ with respect to the reference $\p_0$. When $\p_0 \in \calH_\Gamma^{++}$, this slice will coincide with our joint translation spectrum and the second main result of this article (Theorem \ref{thm.intro.jtspectrum} below) will establish a finer convergence property (compared to Theorem \ref{thm.intro.cone}) of $\D$ and $\bfL$. We now proceed to discuss our second result and this more refined convergence property.

Given $\p \in \calH^{++}_\G$, we define $\p$-normalized versions of $\D$ and $\bfL$ as  
\begin{equation}\label{eq.defDlam}
\D_\p(x) = \frac{1}{\p(o,x)}\cdot \D(x) \ \text{ and } \ \bfL_\p(x) = \frac{1}{\ell_{\p}(x)}\cdot\bfL(x)
\end{equation}
for each $x \in \G$. Since $\p$ is proper, $\D_\p(x)$ is well-defined for all but finitely many elements $x$, and $\bfL_\p(x)$ is defined whenever $x$ is non-torsion. Given $R, T >0$ we let $S_R(T) = \{x \in 
\G: |\p(o,x) - T| \le R \text{ and }\p(o,x)>0\}$ and similarly we 
define $S_R^{\ell_\p}(T) = \{x \in \G : |\ell_\p(x) - T| <R \text{ and }\ell_\p(x)>0 \}$. 

The second main result of this article establishing our geometric version of the joint spectrum and the finer convergence property in the translation cone is the following.

\begin{theorem}[Joint translation spectrum]\label{thm.intro.jtspectrum}
Let $\D=(\p_1,\ldots,\p_n)$ be a finite tuple of elements in $\mathcal{H}_\Gamma$ and $\p \in \mathcal{H}_\Gamma^{++}$. Then, there exists a convex compact set $\J_{\p}(\D) \subset \R^n$ such that for any $R>0$ large enough, both sequences $\D_\p(S_R(T))$ and $\bfL_\p(S_R^{\ell_\p}(T))$ converge to $\J_{\p}(\D)$ in the Hausdorff metric as $T \to \infty$. It has non-empty interior in $\R^n$ if and only if $\p_1,\ldots, \p_n, \p$ are independent.
\end{theorem}

We call the limiting set $\J_{\p}(\D)$ the \textit{joint translation spectrum} of $\p_1, \ldots, \p_n$ with respect to $\p$.

\begin{remark}
    If we study the image of balls $B_R(T) = \{ x\in \G: \p(o,x) < T \}$ (instead of spheres) under $\D_\p$ then we also obtain a limiting object $\J_\p^B(\D)$, which is simply the convex hull of $\J_\p(\D)$ and the zero vector, see \Cref{sec.othersets}. 
\end{remark}






\subsection{Random walk and dynamical spectra}\label{sec.RWandD}
In \cite[Theorem~1.11]{breuillard-sert}, the joint spectrum $\J(S)$ was compared to the \emph{Lyapunov spectrum}, the set of all possible drifts for the Cartan projection associated to random matrix products supported on $S$. In our current setting, we consider the values attained by $\D_\p$ along typical random walks on $\G$.


\begin{definition}\label{def.rw}
    We define the \emph{random walk spectrum}  $\mathcal{WJ}_{\p}(\D)$ to be the collection of vectors $\vx \in \R^n$ for which there exists a finitely supported, symmetric, probability measure $\lam$ on $\G$ with support generating   $\G$ such that the following holds: for almost every trajectory $(Z_k)_k$ of the random walk on $\G$ determined by $\lam$ we have $\D_\p(Z_k) \to \vx$ as $k \to \infty$.
\end{definition}

The joint translation spectrum records the `asymptotic directions' seen by $\D_\p(x)$ for all $x \in \G$, while $\mathcal{WJ}$ records the `asymptotic directions' of $\D_\p$ that are witnessed by a finitely supported random walk on $\G$. When $\p,\p_1,\dots,\p_n\in \calD_\G$, we obtain the following relation between $\J_\p(\D)$ and $\mathcal{WJ}_\p(\D)$.

\begin{theorem}\label{thm.wj=j}
Suppose that $\p,\p_1,\dots,\p_n\in \calD_\G$. Then we have that 
\[
\overline{\mathcal{WJ}_{\p}(\D)} = \J_{\p}(\D).
\]
That is, the closure of the random walk spectrum is the joint translation spectrum. Moreover, if $\p,\p_1,\dots,\p_n$ are independent then 
\[
\mathcal{WJ}_{\p}(\D) \subset \Int(\J_{\p}(\D)).
\]
\end{theorem}


To prove this theorem we first study a different spectrum, that records `asymptotic directions' of $\D_\p$ what are witnessed by Borel probability measures on the Gromov boundary $\partial \G$ of $\G$. A Borel probability measure $\nu$ on $\partial \G$ that is \emph{quasi-invariant} if $\nu$ and $x_\ast \nu$ are absolutely continuous for any $x\in \G$, and it is \emph{ergodic} if either $\nu(A)=0$ or $1$ whenever $A$ is a Borel $\G$-invariant subset of $\partial \G$. 

If $\nu$ is quasi-invariant and ergodic, suppose the vector $\vx$ in $\R^n$ satisfies the following. 
For $\nu$-almost every $\xi\in \partial \G$ and any quasi-geodesic $\x_k$ in $\G$ converging to $\x$ we have that $\D_\p(\x_k)\to \vx$ as $k\to \infty$. In case $\vx$ exists, we denote it by $\D_\p(\nu)$ and we call it the \emph{Lyapunov vector} for $\nu$, in analogy with the matrix case.



\begin{definition}
        We define the \textit{dynamical translation spectrum} $\mathcal{DJ}_{\p}(\D)$ to be the collection of all Lyapunov vectors $\D_\p(\nu)$ for $\nu$ an ergodic and quasi-invariant Borel probability measure on $\partial \G$.        
\end{definition}
The first step in the proof of \Cref{thm.wj=j} is to relate the joint translation spectrum with its dynamical counterpart. We will prove the following.
\begin{theorem}\label{thm.dj=j}
We have that
\[
\overline{\mathcal{DJ}_{\p}(\D)} = \J_{\p}(\D).
\]
That is, the closure of the dynamical translation spectrum is the joint translation spectrum. Moreover, if $\p_1,\dots,\p_n,\p$ are independent then $\calD\J_\p(\D)$ contains the interior of $\J_\p(\D)$.
\end{theorem}


\subsection{Manhattan manifolds}\label{sec.mm}
To link the joint translation spectrum with its random walk and dynamical counterparts, we study Manhattan manifolds. These are higher dimensional versions of the well-known Manhattan curve introduced by Burger \cite{burger}.

Let $\D, \p$ be as above and let $\langle \cdot, \cdot \rangle$ denote the standard inner product on $\R^n$. We consider the following set
\[
\left\{ (v_1, \ldots, v_{n+1}) \in \R^{n+1} : \sum_{x \in \G} e^{-\langle (v_1,\ldots, v_n), \, \D(x) \rangle - v_{n+1} \p(o,x)} < \infty \right\},
\]
which is convex and has convex boundary. We are interested in the boundary of this set.

\begin{definition}
Given $\vv \in \mathbb{R}^n$ we define $\theta_{\D/\p}(\vv)$ to be the critical exponent of the series
\[
s \mapsto \sum_{x\in\G} e^{-\langle \vv, \D(x) \rangle - s\p(o,x)}
\]
as $s$ varies in $\R$. Then the \emph{Manhattan manifold} is the graph of $\theta_{\D/\p}$. That is,
\[
\calM_{\D/\p}:=\{(\vv,y) \in \R^n \times \R \text{ such that } y = \theta_{\D/\p}(\vv)\}.
\]
\end{definition}
Combining ideas of Burger \cite{burger}, Cantrell-Tanaka \cite{cantrell-tanaka.manhattan} and Cantrell-Reyes \cite{cantrell-reyes.manhattan} we prove the following.

\begin{theorem}\label{thm.manc^1} 
The function $\theta_{\D/\p}$ is of class $C^1$. Moreover, if $\p_1,\ldots, \p_n, \p$ are independent then $\thet_{\D/\p}$ is strictly convex, i.e. $\nabla \theta_{\D/\p}$ is injective. In this case the Manhattan manifold $\calM_{\D/\p}$ is an $n$-dimensional, $C^1$-manifold.
\end{theorem}


In the case that $\p_1,\ldots, \p_n, \p$ are not independent, we can take a maximal independent subset $\p, \p_1,\ldots, \p_n$ that includes $\p$ and then $\calM_{\D/\p}$ is a $k$-dimensional $C^1$-manifold. 

The relation between the Manhattan manifold and the joint translation spectrum is via the gradient of the parametrization $\thet_{\D/\p}$. Indeed, we obtain a complete characterization of the interior points in $\J_\p(\D)$.

\begin{theorem}\label{thm.homeointeriors}
    Suppose that $\p_1,\dots,\p_n,\p$ are independent. Then 
    \[\Int(\J_\p(\D))=\{-\nabla \thet_{\D/\p}(\vv): \vv\in \R^n\}\]
    and the map $-\nabla\thet_{\D/\p}:\R^n \ra \Int(\J_\p(\D))$ is a homeomorphism.
\end{theorem}

This result allows us to give more precise information about some special vectors in $\calD\J_\p(\D)$. The \emph{mean distortion} vector for $\p_1,\dots,\p_n$ with respect to $\p$ is the limit
\[
\tau_{\D/\p} = \lim_{T\to\infty} \frac{1}{\#S_R(T)}\sum_{x \in S_R(T)} \D_\p(x).
\]
This vector is well-defined and can be seen as the `typical quasi-isometry constant' vector for the collection $\p_1,\ldots, \p_n, \p$. The mean distortion vector is also the Lyapunov vector $\D_\p(\nu_\p)$ for $\nu_\p$ a quasi-conformal measure for $\p$ (see \Cref{sec.boundary} and \cite[Theorem~3.12]{cantrell-tanaka.manhattan}).

As in Breuillard-Sert's work \cite[Theorem~1.9]{breuillard-sert}, from \Cref{thm.homeointeriors} we will deduce that $\tau_{\D/\p}$ sits in the interior of the joint translation spectrum.

\begin{proposition}\label{prop.interiorlambda}
Suppose that $\p_1,\ldots, \p_n, \p$ are independent. Then  $\tau_{\D/\p}=-\nabla \theta_{\D/\p}(\mathbf{0})$, and consequently $\tau_{\D/\p}$ lies in the interior of $\J_{\p}(\D)$.
\end{proposition}

For $\p_1,\ldots, \p_n,\p$ not necessarily independent, from this proposition we get that the vector $\tau_{\D/\p}$ lies in the relative interior of $\J_\p(\D)$, in analogy to \cite[Theorem~1.9]{breuillard-sert}.

\subsection{Manhattan manifolds as metric structures and its boundary}\label{subsec.intro.manhattanmetric}

In the work of Cantrell-Reyes \cite{cantrell-reyes.manhattan}, Manhattan curves for pairs of metrics in $\calD_\G$ were used to describe geodesics in the space of \emph{metric structures} $\scrD_\G$, which is the quotient of $\calD_\G$ under rough similarity (see \cite{oregonreyes.metric} for the metric properties of $\scrD_\G$). In the current case, similar techniques allow us to embed the Manhattan manifold inside $\scrH_\G^{++}$, the set of rough similarity classes in $\calH_\G^{++}$. 

More precisely, if $\thet=\thet_{\D/\p}$ then for every value $\vv$ we let $\rho^{\D/\p}(\vv)=\rho^\thet(\vv)$ denote the rough similarity class of $\p_{\vv}:=\langle \vv,\D\rangle+\theta_{\D/\p}(\vv)\p$, which belongs to $\calH_\G^{++}$ by \Cref{lem.EGR1}. We let $\scrM_{\D/\p}$ denote the image of $\rho^\thet$, which we also call \emph{Manhattan manifold}. The metric on $\scrD_\G$ naturally extends to a metric on $\scrH_\G^{++}$, and for its induced topology we have the following complement to \Cref{thm.homeointeriors}.


\begin{proposition}\label{prop.fromManhattantoJTS}
Suppose that $\p_1,\dots,\p_n, \p$ are independent. Then the map $\rho=\rho^{\D/\p}:\R^n \ra \scrH_\G^{++}$ is a homeomorphism onto its image. In consequence, the composition
\[ \scrM_{\D/\p}   \xrightarrow{\rho^{-1}} \R^n \xrightarrow{-\nabla \thet_{\D/\p}} \Int(\J_\p(\D))\] 
is a homeomorphism.
\end{proposition}

As for Manhattan curves in $\calD_\G$, we can also define a bordification of $\scrM_{\D/\p}$ in terms of `boundary' hyperbolic metric potentials. If $\ov\calH^{++}_\G$ is the set of hyperbolic metric potentials bounded below and with non-trivial stable translation length function, we let $\partial \calH_\G^{++}=\ov\calH^{++}_\G \bs \calH_\G^{++}$. Then $\ov\scrH_\G^{++}$ and $\partial \scrH_\G^{++}$ denote the rough similarity quotients of $\ov\calH_\G^{++}$ and $\partial \calH_\G^{++}$ respectively. If $[\vv]$ is a unit vector in $\mathbb{S}^{n-1}=(\R^n-\{\mathbf{0}\})/\R^+$, then we can use $\thet_{\D/\p}$ to define a boundary metric structure $\rho^\thet([\vv])\in \partial\scrH_\G^{++}$ (see \Cref{def.manhattanmanifoldmetric}). If $\ov{\R^{n}}=\R^n \cup \mathbb{S}^{n-1}$ is the ball compactification of $\R^n$, then we obtain an extension map $\rho: \ov{\R^n} \ra \ov\scrH_\G^{++}$ whose image we denote by $\ov\scrM_{\D/\p}$. We will see in \Cref{lem.extensionboundary} that this map is injective.

\begin{remark}\label{rmk.topologyboundary}
Inducing the appropriate topology on $\ov\scrH_\G^{++}$, it can be proven that the map $\rho:\ov{\R^n} \ra \ov\scrM_{\D/\p}$ is actually a homeomorphism. However, we will not pursue this in this paper.   
\end{remark}

The question of whether the homeomorphism $-\nabla\thet_{\D/\p}:\R^n \ra \Int(\J_\p(\D))$ from \Cref{thm.homeointeriors} extends to the boundaries is more subtle, and it depends on the regularity properties of the boundary of $\J_\p(\D)$. We obtain the following criteria.

\begin{proposition}\label{prop.descriptionboundary}
    If $\p_1,\dots,\p_n$ and $\p$ be independent then the following hold:
    \begin{enumerate}
        \item If $\J_\p(\D)$ is strictly convex then the map $-\nabla \thet_{\D/\p}:\R^n \ra \Int(\J_\p(\D))$ extends to a continuous surjection 
        $\Phi:\ov{\R^n} \ra \J_\p(\D)$.
\item If $\partial \J_\p(\D)$ is $C^1$ then the map $(-\nabla \thet_{\D/\p})^{-1}:\Int(\J_\p(\D)) \ra \R^n$ extends to a continuous surjection $\Psi:\J_\p(\D) \ra\ov{\R^n}$.
    \end{enumerate}
\end{proposition}


\subsection{Examples}

The flexibility in the definition of $\calH_\G$ gives us plenty of examples of joint translation spectra, and we briefly discuss some of them. For more details, see \Cref{sec.examples}.

We first observe a case in which the joint translation spectrum is actually a matrix joint spectrum. Let $\G$ be a hyperbolic surface group and let $\rho: \G \ra \PSL_n(\R)$ be a \emph{Hitchin representation}. That is, a  deformation of Fuchsian representation into $\PSL_2(\R)$ composed with an irreducible representation $\PSL_2(\R)\ra \PSL_n(\R)$. We let $\D_{\rho}=(\p_1,\dots,\p_n)$ where $\p_j(x,y)=\log \sig_j(\rho(x^{-1}y))$. Then $\p_1,\dots,\p_n \in \calH_\G$. If $S \subset \G$ is a finite subset generating $\G$ as a semigroup, then we can check the identity
$$\J(\rho(S))=\J_{d_S}(\D_\rho).$$
Moreover, the translation cone $\calT\calC(\D_\rho)$ is precisely the Benoist limit cone of the representation $\rho$.

From a purely metric point of view, natural examples to consider are when the hyperbolic metric potentials are word metrics on $\G$ with respect to finite and symmetric generating sets. In this case, the joint translation spectrum is always a polytope. The same happens when the input metrics are induced by geometric actions on $\CAT(0)$ cube complexes equipped with their combinatorial metrics. 


As an example not coming from metrics, we have the unit ball for the \emph{stable norm} in homology associated to points in Teichm\"uller space of a closed hyperbolic surface \cite{massart,pollicott-sharp.livsic}. Under the appropriate identification, this unit ball is the joint translation metric for a basis of $H^1(\G;\Z)$ as input potentials, together with the corresponding point in Teichm\"uller space. This joint translation spectrum is not a polytope.




\subsection*{Organization of the paper}
The article is organized as follows. In \Cref{sec.preliminaries} we review some preliminaries about hyperbolic spaces and groups, geodesic currents, and the Mineyev flow. Then in \Cref{sec.H_G} we study further the space of hyperbolic metric potentials and discuss several examples. The proofs of our main results begins in \Cref{sec.constructionJTS}, in which we prove the existence of the joint translation spectrum and prove Theorems \ref{thm.intro.cone} and \ref{thm.intro.jtspectrum}. The Manhattan manifold is then studied in \Cref{sec.manhattanmanifolds}. There we prove Theorems \ref{thm.wj=j}, \ref{thm.dj=j} and \ref{thm.manc^1}, and one of the assertions of \Cref{thm.homeointeriors}. We continue with \Cref{sec.metricstructuresmanhattan} in which we project the Manhattan manifold into the space of metric structures, proving Proposition \ref{prop.fromManhattantoJTS} and \ref{prop.descriptionboundary} and completing the proof of \Cref{thm.homeointeriors}. In \Cref{sec.examples} we discuss several examples of joint translation spectra for some relevant hyperbolic metric potentials, and we end with some questions in \Cref{sec.questions}.




\subsection*{Acknowledgements}
E.~R. would like to thank the Max Planck Institut f\"ur Mathematik for its hospitality and financial support.


\section{Preliminaries}\label{sec.preliminaries}

In this section we review some metric and measure-theoretic aspects of hyperbolic groups. For more details we refer the reader to \cite{bridson-haefliger,ghys-delaharpe}.

\subsection{Metric notions} \label{sec.metricnotions}

Given a pseudo metric space $(X,d)$, the \emph{Gromov product} $(\cdot| \cdot)_\cdot:X^3 \ra \R$ is defined according to
\begin{equation}\label{eq.defGroprod}
   (x|y)_w = (x|y)_w^d:= \frac{1}{2}(d(x,w) + d(w,y) - d(x,y)) \text{ for $x,y,w \in X$.} 
\end{equation}
The pseudo metric $d$ is $\del$-\emph{hyperbolic} if
\[
(x|z)_w \ge \min\{(x|y)_w, (y|z)_w \} - \delta \text{ for all $x,y,z,w \in X$,}
\]
and $d$ is \emph{hyperbolic} if it is $\del$-hyperbolic for some $\del$. 


Given an interval $I \subset \R$ and constants $L, C >0$ we say that a map $\g: I \to X$ is an \textit{$(L, C)$-quasi-geodesic} if 
\[
L^{-1} \, |s-t|-C \le d(\g(s), \g(t)) \le L\, |s-t|+C \  \text{ for all $s, t \in I$},
\]
and a \textit{$C$-rough geodesic} if
\[
|s-t|-C \le d(\gamma(s), \gamma(t))\le |s-t|+C \ \text{ for all $s, t \in I$}.
\]
A $0$-rough geodesic is referred to as a \emph{geodesic}. When we want to emphasize the dependence on $d$ we will use the terms $(L,C,d)$-quasi-geodesics and $(C,d)$-rough geodesics. The pseudo metric space $(X, d)$ is called $C$-\emph{roughly geodesic} if for each pair of elements $x, y \in X$ we can find a $C$-rough geodesic with extreme points $x$ and $y$. We say that $(X,d)$ is \textit{roughly geodesic} if it is $C$-roughly geodesic for some $C \ge 0$. 

Given pseudo metric spaces $(X,d)$ and $(X_\ast,d_\ast)$, the map $F:X \ra X_\ast$ is a \emph{quasi-isometric embedding} if there exist constants $\lam_1,\lam_2,C>0$ such that 
\begin{equation}\label{eq.defqi}
    \lam_1 d(x,y)-C\leq d_\ast(Fx,Fy) \leq \lam_2d(x,y) d+C \ \text{ for all}x,y\in X. 
\end{equation}
If there is also some $C'$ such that every $x_\ast\in X_\ast$ is within $C'$ from some element in $F(X)$ we say that $F$ is a \emph{quasi-isometry}. A quasi-isometry $F$ satisfying \eqref{eq.defqi} with $\lam_1=\lam_2$ is called a \emph{rough similarity}, and it is called a \emph{rough isometry} when $\lam_1=\lam_2=1$. Two pseudo metrics $d,d_\ast$ on a space $X$ are called \emph{quasi-isometric} (resp. \emph{roughly similar/isometric}) if the identity map $(X,d) \ra (X,d_\ast)$ is a quasi-isometry (resp. rough similarity/isometry).

\subsection{Hyperbolic groups and metric structures}\label{sec.hyp}

Throughout this work $\G$ will be a \emph{non-elementary hyperbolic group}: a finitely generated group $\G$, that is not virtually cyclic and such that any word metric $d_S$ on $\G$ associated to a finite, symmetric generating set $S$ is hyperbolic. We use $o$ to denote the identity element of $\G$. 

We write $\Dc_\G$ for the set of all the pseudo metrics on $\G$ that are hyperbolic, quasi-isometric to a word metric (for a finite generating set), and $\G$-invariant. It is known that every pseudo metric in $\Dc_\G$ is roughly geodesic. 

The \emph{stable translation length} of $d \in \Dc_\G$ is the function
\begin{equation*}
    \ell_d(x):=\lim_{m\to \infty}{\frac{1}{m}d(o,x^m)} \hspace{2mm} \text{for } x\in \G,
\end{equation*}
which is well-defined by subadditivity. The \emph{exponential growth rate} of $d$ is the quantity 
\[
v_d = \limsup_{T\to\infty} \frac{1}{T} \log \#\{x \in \G: d(o,x) < T\} \in \R^+.
\]


Given $d,d_\ast \in \G$, the \emph{dilation} of $d_\ast$ with respect to $d$ is the quantity
\[
\Dil(d_\ast,d) = \sup_{x} \frac{\ell_{d_\ast}(x)}{\ell_{d}(x)},
\]
where the supremum is taken over all the non-torsion elements $x\in \G$. Note that $\Dil(d_\ast,d)$ is positive and finite, and greater or equal to 1 when $v_d=v_{d_\ast}=1$ \cite[Lemma 3.6]{oregonreyes.metric}. Also, we have that $d,d_\ast$ are roughly similar if and only
if $\Dil(d,d_\ast)\Dil(d_\ast,d) = 1$. From this observation we define the space of \emph{metric structures} $\scrD_\G$ to be the space $\Dc_\G$ modulo the equivalence relation of rough similarity which we equip with the (symmetrized) \emph{Thurston metric}
\begin{equation}\label{eq.defDel}
 \Del([d],[d_*]):=\log  \left(\Dil(d,d_*) \Dil(d_*,d) \right).  
\end{equation}

Here $[d]\in \scrD_\G$ denotes the rough similarity class of $d\in \calD_\G$. The geometry of $(\scrD_\G,\Del)$ was studied by the second author and the first two authors in \cite{oregonreyes.metric} and \cite{cantrell-reyes.manhattan}.

Natural metrics in $\calD_\G$ are word metrics for finite and symmetric generating subsets of $\G$. More generally, if $\G$ acts properly and cocompactly by isometries on  the geodesic metric space $(X,d)$ and $p\in X$ then $d_X^p(x,y):=d(x\cdot p,y \cdot p)$ defines a metric belonging $\calD_\G$.

Another important class of metrics in $\Dc_\G$ are \textit{Green metrics}. Let $\lam$ be a finitely supported, symmetric, probability measure on $\G$ such that the support of $\lam$ generates $\G$ as a semi-group. Such a measure is called \emph{admissible}. The Green's function associated to $\lam$ is the function
\[
G_\lam(x,y) = \sum_{k\ge 0} \lam^{\ast k}(x^{-1}y) \ \ \text{ for $x,y \in \G$}
\]
from which we define the Green metric
\[
d_\lam(x,y): = - \log\left(\frac{G_\lam(x,y)}{G_\lam(o,o)} \right) \ \ \text{ for $x,y \in \G$.}
\]
Intuitively $d_\lam$ assigns the distance between $x,y \in \G$ to be the (minus logarithm of the) probability that a $\lam$-random walk starting at $x$ reached $y$. Two key properties of Green metrics to note are that
\begin{enumerate}
    \item they have exponential growth rate $1$ \cite{BHM-greenspeed}; and,
    \item the quasi-conformal measures for the metric $d_\lam$ are in the same class as the hitting measure for $\lam$ in the Gromov boundary $\partial \G$ (see \Cref{sec.boundary} and \cite{BHM.harmonic}).
\end{enumerate}

Green metrics also serve as input metrics to construct Mineyev's topological flow, see \Cref{sec.mf}. In addition, the following property of Green metrics is key in the proof of \Cref{thm.wj=j}, and is proven in the forthcoming work of the first two authors and D\'idac Mart\'inez-Granado \cite{CMGR}.
\begin{theorem}\label{thm.gmd}
    Green metrics are dense in the space $(\scrD_\G, \Del)$. That is, the collection of equivalence classes in $\scrD_\G$ that contain a Green metric are dense with respect to the metric $\Del$. 
\end{theorem}


\subsection{The Gromov boundary and tempered potentials}\label{sec.boundary}

Hyperbolic groups can be compactified using their Gromov boundary $\partial \G$ which consists of equivalence classes
of divergent sequences. We fix a reference metric $d$ in $\Dc_\G$ and say that sequence of group elements $\{x_n\}_{n=0}^\infty$ \emph{diverges}
if $(x_n|x_m)_o$ (computed with respect to $d$) tends to infinity as $\min\{n, m\}$ tends to infinity. Two divergent sequences $\{x_n\}_{n=0}^\infty$ and $\{y_n\}_{n=0}^\infty$ are \emph{equivalent} if $(x_n|y_m)_o$ tends to infinity as $\min\{n, m\}$ tends to infinity. The notions of divergence and equivalence are independent of the reference metric $d$. The \emph{Gromov boundary} of $\G$ is the space $\partial \G$ of equivalence classes of divergent sequences in $\G$. If $\xi \in \partial \G$ is represented by the divergent sequence $\{x_n\}_n$ we say that $x_n$ \emph{converges} to $\xi$.

As in \eqref{eq.defGroprod}, given a function $\psi:\G \times \G \ra \R$ its Gromov product is given by 
\begin{equation*}
   (x|y)_w^\p:= \frac{1}{2}(\p(x,w) + \p(w,y) - \p(x,y)) \text{ for $x,y,w \in \G$.} 
\end{equation*}

We are interested in the class of \emph{tempered potentials}, introduced and studied in \cite{cantrell-tanaka.manhattan}. Given $\psi: \G \times \G \ra \R$, we say that $(\cdot | \cdot)^\p_o$ admits a `quasi-extension' to $\G \cup \partial \G$ if it can be extended to $(\cdot | \cdot)^\p_o: (\G \cup \partial \G)^2 \to \R$ in such a way that for some constant $C\ge 0$ we have that
\begin{equation}\label{Eq:QE}\tag{QE}
\limsup_{n \to \infty}(x_n|y_n)^\p_o-C \le (\x|\y)^\p_o \le \liminf_{n \to \infty}(x_n'|y_n')^\p_o+C
\end{equation}
for all $\x, \y \in \G \cup \partial \G$ and for all $\{x_n\}_{n=0}^\infty, \{x_n'\}_{n=0}^\infty \in \x$ and $\{y_n\}_{n=0}^\infty, \{y_n'\}_{n=0}^\infty \in \y$ (if $\x\in \G$ we interpret $\{x_n\}_{n=0}^\infty\in \x$ as the sequence $\{x_n\}_n$ being eventually constant and equal to $\x$).

We are also interested in the behavior of $\p$ along rough geodesics with respect to a reference pseudo metric $d\in \calD_\G$. We say that $\p$ satisfies the `rough geodesic' condition if for all large enough $C, R \ge 0$, there exists $C_0 \ge 0$ such that for all $x,y\in \G$ and any $(C,d)$-rough geodesic $\g$ between $x$ and $y$, we have
\begin{equation}\label{Eq:RG}\tag{RG}
|(x| y)_w^\p| \le C_0
\end{equation}
for all $w$ in the $R$-neighborhood of $\g$ (with respect to the metric $d$). The Morse lemma implies that the rough geodesic condition is independent of the reference metric $d$.

\begin{definition}\label{def.tempered}
   A function $\psi: \G\times \G \to \R$ is a \emph{tempered potential} if it is invariant under the left-action of $\G$, and satisfies both the quasi-extension condition \eqref{Eq:QE} and the rough geodesic condition \eqref{Eq:RG}.
\end{definition}

\begin{remark}\label{rmk.invariancefortempered}
  Our definition of tempered potential is slightly different than that in \cite{cantrell-tanaka.manhattan} since we require $\G$-invariance.  
\end{remark}

Hyperbolicity and the Morse lemma imply that any pseudo metric in $\calD_\G$ is a tempered potential. Moreover, linear combinations of tempered potentials are again tempered potentials.

Given $R>0$ and a reference metric $d\in \G$ with quasi-extension $(\cdot|\cdot)_o^d:(\G \cup \partial \G)^2 \ra \R$, the \textit{shadow set} of $x \in \G$ is
\[
O(x,R) =O_d(x,R) =\{ \x \in \partial \G : (x|\x)^d_o \ge d(o,x) - R\} \subset \partial \G.
\]

If $\psi$ is a tempered potential, then its \emph{Busemann function} for $\b^\psi_o:\G \times \partial \G$ is defined according to
\begin{equation}\label{eq.defbusemann}
  \b^\psi_o(x, \x):=\sup\big\{\limsup_{n \to \infty}(\psi(x, \x_n)-\psi(o, \x_n)) \ : \ \{\x_n\}_{n=0}^\infty \in \x\big\}.
\end{equation}

For $d\in \calD_\G$, we note that $\b^d_o$ is the usual Busemann function on $(\G,d)$.  The Busemann function depends on the chosen quasi-extension of the Gromov product of $\p$, but any two such extensions will give Busemann functions that differ by a uniformly bounded function. In addition, there exists a constant $C'$ such that for a large enough $R$ and for all $x \in \G$,
\begin{equation}\label{Eq:corRG}
|\b_o^\psi(x, \x)+\psi(o, x)| \le C' \quad \text{for all $\x \in O(x, R)$}.
\end{equation}
Furthermore, $\b^\psi_o$ is a \emph{quasi-cocycle}. That is, we can find $C''>0$ such that
\begin{equation}\label{eq.busemanncocycle}
   |\b^\psi_o(xy, \x)-(\b^\psi_o(y, x^{-1}\x)+\b^\psi_o(x, \x))| \le C'' \text{ for all } x,y\in \G ,\x \in \partial \G. 
\end{equation}

\subsection{Geodesic currents}

The Gromov boundary $\partial \G$ has a natural topology making it into a compact metrizable space. The \emph{double boundary} is the set $\partial ^2 \G$ of ordered pairs of distinct points of $\partial \G$. We equip $\partial ^2 \G$ with the topology induced by the inclusion $\partial ^2 \G \subset (\partial\G)^2$. With this topology the diagonal action of $\G$ on $\partial ^2 \G$ is topological and cocompact. 

A \emph{geodesic current} on $\G$ is a $\G$-invariant Radon measure on $\partial ^2 \G$. We let $\calC_\G$ denote the space of all the geodesic currents on $\G$ equipped with the weak$^\ast$ topology. We also let $\bbP\calC_\G=(\calC_\G-\{0\})/\R^+$ denote the space of \emph{projective geodesic currents} equipped with the quotient topology. The space $\bbP\calC_\G$ is compact and metrizable.

\begin{remark}\label{rmk.flipinvariant}
Some authors require geodesic currents to be invariant under the flip $(\x,\eta) \to (\eta,\x)$. However, we will not make this assumption in this work. Indeed, we are also interested in geodesic currents that are not flip invariant. 
\end{remark}

Bonahon introduced geodesic currents in \cite{bonahon.currentshpygroups} as a completion of the space of conjugacy classes in $\G$. Indeed, if $x$ is a primitive non-torsion element (i.e. if $x=y^m$ for some $y\in \G$ then $|m|=1$ and $x^+,x^-$ are its attracting and repelling fixed points in $\partial \G$ respectively, we let $\mu_x$ be the sum of the atomic measures supported on the $\G$-translates of $(x^-,x^+)$. If $x=y^m$ for some $y\in \G$ primitive and $m\in \N$, we set $\mu_x:=m\mu_y$. Clearly $\mu_x$ is discrete, hence defines a geodesic current: the \emph{rational current} associated to $x$. Bonahon \cite{bonahon.currentshpygroups} proved that real multiples of rational currents are dense in $\calC_\G$.

As noted by Furman \cite{furman}, pseudo metrics in $\calD_\G$ also induce geodesic currents in a natural way. Given $d \in \Dc_\G$, a Borel probability measure $\nu$ on $\partial \G$ is \emph{quasi-conformal} for $d$ there exists a constant $C>1$ such that for every $x\in \G$ and $\nu$-almost every $\x\in \partial \G$ we have
\begin{equation}\label{eqqcmeasdef}
     C^{-1}e^{-v_d\b_o^d(x,\x)}\leq \frac{dx_\ast\nu}{d\nu}(\x)\leq Ce^{-v_d\beta_o^d(x,\x)},
\end{equation}
where $\beta_o^d$ is a Busemann quasi-cocycle for $d$ and $v_d$ is the exponential growth rate of $d$. Quasi-conformal measures exist for every $d\in \calD_\G$ \cite{coornaert}, and any two quasi-conformal measures for $d$ are absolutely continuous with uniformly bounded Radon-Nikodym derivatives. Moreover, for $d,d_\ast\in \calD_\G$ and quasi-conformal measures $\nu,\nu_\ast$ for $d,d_\ast$ respectively, we have that $\nu$ and $\nu_\ast$ are absolutely continuous with respect to each other if and only if $d$ and $d_\ast$ are roughly similar.

If $\nu$ is a quasi-conformal measure for $d$ then there exists a geodesic current $\Lam_d$ in the measure class of $\nu \otimes \nu$. Indeed, it is of the form
\begin{equation}\label{eq.defBM}
  d\Lam_d(\x,\y)=\al(\xi,\y)\exp(2 v_d(\x|\y)^d_o)d\nu (\x)d\nu(\y)  
\end{equation}
for a measurable function $\al$ that is essentially bounded and bounded away from zero. Any geodesic current satisfying \eqref{eq.defBM} is ergodic ($\G$-invariant Borel subsets have either zero or full measure), and any two such currents must differ by a scaling factor. Hence the projective class of $\Lam_d$ depends only on the metric structure $[d]\in \scrD_\G$. We call any such $\Lam_d$ a \emph{Bowen-Margulis current} for $d$ (or $[d]$). 





\subsection{The Mineyev flow}\label{sec.mf}
In this section, we follow the discussion in \cite[Section 2]{cantrell-tanaka.manhattan}. We fix a \emph{strongly hyperbolic} metric $\wh d \in \calD_\G$ \cite{nica-spakula}. For example, $\wh d$ can be a Green metric   
associated to a finitely supported, symmetric, admissible probability measure on $\G$ (see Section 2 of \cite{cantrell-tanaka.manhattan}). We will follow Mineyev's construction \cite{mineyev.flow} to obtain a coarse geometric analogue of a geodesic flow associated to the metric $\wh d$. Mineyev originally constructed the flow using the so-called Mineyev hat metric, but any strongly hyperbolic metric has the same required properties for the construction of the flow. The key fact is that a strongly hyperbolic metric $\wh d$ has a Busemann function $\wh \b_o=\beta_o^{\wh d}$ that is a continuous \emph{cocycle}, in the sense that it is continuous on $\G \times \partial \G$ and satisfies \eqref{eq.busemanncocycle} with $C''=0$. 

Now note that there is a constant $C\ge 0$
such that for each $(\x, \y) \in \partial^2 \G$
there is a $C$-rough geodesic $\g_{\x, \y}: \R \to (\G, \wh d\,)$ with endpoints $\x$ and $\eta$ at $-\infty$ and $\infty$ respectively \cite[Proposition 5.2 (3)]{bonk-schramm}.
We parameterize $\g_{\x, \y}$ in such a way that
\[
\wh d(\g_{\x, \y}(0), o)=\min_{t \in \R}\wh d(\g_{\x, \y}(t), o)
\]
and define
\[
\ev: \partial^2 \G \times \R \to \G \ \ \text{ by } \ \ \ev(\x, \y, t):=\g_{\x, \y}(t).
\]
Note that the map $\ev$ depends on the choice of $C$-rough geodesics,
however, every other choice yields the map whose image lies in a uniformly bounded distance by hyperbolicity. Equip the space of $C$-rough geodesics on $(\G, \wh d\,)$ with the pointwise convergence topology and define $\ev: \partial^2 \G \times \R \to \G$ as a measurable map by assigning $\g_{\x, \y}$ to $(\x, \y) \in \partial^2 \G$ in a Borel measurable way. To do this first fix a set of generators $S$ in $\G$ and an order on it. Then consider $C$-rough geodesics evaluated on the set of integers as sequences of group elements and choose lexicographically minimal ones $\g^0_{\x, \y}$ for each $(\x, \y) \in \partial^2 \G$. Finally we define $\g_{\x, \y}(t):=\g^0_{\x, \y}(\lfloor t\rfloor)$ for $t \in \R$ where $\lfloor t\rfloor$ is the largest integer at most $t$.

Letting $\wh \b_o: \G \times \partial \G \to \R$ be the Busemann function based at $o$ associated with $\wh d$, we define the cocycle $\k:\G \times \partial^2 \G \to \R$ by
\[
\k(x, \x, \y):=\frac{1}{2}\left(\wh \b_o(x^{-1}, \x)-\wh \b_o(x^{-1}, \y)\right)
\]
where the cocycle identity for $\k$ follows from \eqref{eq.busemanncocycle} (recall that $C''=0$ in this case). Then $\G$ acts on $\partial^2 \G \times \R$ through $\k$ via
\[
x\cdot(\x, \y, t):=(x \x, x\y, t-\k(\x, \y, t)) \ \text{ for $x\in\G$.}
\]
This is called the {\it $(\G, \k)$-action} on $\partial^2 \G \times \R$.
Tanaka showed that the $(\G, \k)$-action on $\partial^2 \G \times \R$ is properly discontinuous and cocompact \cite{tanaka.topflows}.
That is, the quotient topological space $\calF_\k:=\G \backslash (\partial^2 \G \times \R)$ is compact.
We define an $\R$-action $\wt \F$ on $\partial^2 \G \times \R$ by
translation in the $\R$-component:
\[
\wt \F_s(\x, \y, t):=(\xi, \y, t+s).
\]
This action and the $(\G, \k)$-action commute and thus
the $\R$-action $\wt \F$ descends to the quotient
\[
\F_s[\x, \y, t]:=[\x, \y, t+s] \quad \text{for $[\x, \y, t] \in \Fc_\k$}.
\]
Then $\R$ acts on $\Fc_\k$ via $\F$ continuously.
We call the $\R$-action $\F$ on $\Fc_\k$ the {\it topological flow} (or, simply the {\it flow}) on $\Fc_\k$.

We are interested in Borel measures that are invariant under the flow on $\Fc_\k$.
Given a geodesic current $\L\in \calC_\G$ there exists a unique finite Radon measure $m=m_\Lam$ invariant under the flow on $\Fc_\k$ and
such that
\[
\int_{\partial^2 \G\times \R}f\,d\L dt=\int_{\Fc_\k}\wbar f\,dm
\]
for all compactly supported continuous functions $f$ on $\partial^2 \G \times \R$,
where $\wbar f$ is the $\G$-invariant function
\[
\wbar f(\x, \y, t):=\sum_{x \in \G}f(x\cdot(\x, \y, t)),
\]
considered as a function on $\Fc_\k$ \cite[Lemma 3.4]{tanaka.topflows}. 
Often we normalize $\L$ in such a way that the corresponding flow invariant measure $m_\L$ has
the total measure $1$ (and so is a probability measure on $\Fc_\k$). 

Given $d \in \Dc_\G$ and a (normalized) Bowen-Margulis current $\Lam_d$ for $d$, 
we let $m_d=m_\L$ denote the corresponding flow invariant probability measure on $\Fc_\k$. This measure is ergodic with respect to the flow, in the sense that for every Borel set $A$ such that $\F_{-t}(A)=A$ for all $t \in \R$,
either $m_d(A)=0$ or $1$ 
\cite[Proposition 2.11 and Theorem 3.6]{tanaka.topflows}.


\section{Hyperbolic metric potentials} \label{sec.H_G}

In this section we explore further properties of the class $\calH_\G$ of hyperbolic metric potentials introduced in \Cref{sec.jts}. Recall that these are functions $\p:\G \times \G \ra \R$ that are $\G$-invariant and such that for any $d_0\in \calD_\G$ there exists $\lam>0$ satisfying  
\begin{equation}\label{eq.defH_G}
   |(x|y)_{w}^\psi| \leq \lam (x|y)_w^{d_0}+\lam 
\end{equation}
for all $x,y,w \in \G$. In particular, by comparing \eqref{eq.defH_G} and \eqref{Eq:RG} we see that any tempered potential belongs to $\calH_\G$. The converse also holds, namely:

\begin{proposition}\label{prop.H_G=temp}
    $\calH_\G$ coincides with the space of tempered potentials on $\G$. 
\end{proposition}
We postpone the proof of this proposition after reviewing some properties of hyperbolic metric potentials. First, we note that $\calH_\G$ is a real vector space containing the space $\calB_\G$ of bounded and $\G$-invariant functions from $\G \times \G$ into $\R$. Also, $\calH_\G$ contains the space $\widehat\calD_\G$ of \emph{hyperbolic distance-like functions} on $\G$ \cite[Definition 3.1]{cantrell-reyes.manhattan}. A hyperbolic distance-like function is a nonnegative and $\G$-invariant map $\p:\G \times \G \ra \R$ satisfying \eqref{eq.defH_G}, $\p(x,x)=0$ for all $x\in \G$, and the triangle inequality $(x|y)_w^\p\geq 0$ for all $x,y,w\in \G$. In particular, we have the inclusions $\calD_\G \subset \ov\calD_\G \subset \wh \calD_\G$, where $\ov\calD_\G$ is the  space of all the $\G$-invariant pseudo metrics on $\G$ satisfying \eqref{eq.defH_G}.

\subsection{Structural properties}
As mentioned in the introduction, the \emph{stable translation length} of $\p\in \calH_\G$ is the function $\ell_\psi: \G \ra \R$ given by
\begin{equation}\label{eq.defSTLcalH}
    \ell_\psi(x):=\lim_{k \to \infty}{\frac{\psi(o,x^k)}{k}}.
\end{equation}

We first check that this function is well-defined.

\begin{lemma}\label{lem:MLSnonegative}
    For any $\psi\in \calH_\G$ and $x\in \G$ the limit in \eqref{eq.defSTLcalH} exists. Moreover, the following two are equivalent:
    \begin{enumerate}
    \item there exists $C>0$ such that $\psi(x,y)\geq -C$ for all $x,y\in \G$.
    \item $\ell_\psi(x)\geq 0$ for all $x\in \G$.     
    \end{enumerate}
\end{lemma}

\begin{proof}
Fix a metric $d \in \Dc_\G$ and let $x\in \G$.
Clearly $\ell_\psi(x)$ is well-defined and equal to $0$ when $x$ is torsion, so suppose now that $x$ is non-torsion.   Note that 
    \begin{equation}\label{eq.gpineq}
        \psi(o,x^{n+m}) = \psi(o,x^n) + \psi(o,x^m) - 2(x^{-n}|x^{m})_o^\psi
    \end{equation}
    for all $n,m \ge 1$ and $x \in \G$.  From \cite[Lemma 3.11]{cantrell-reyes.manhattan} it follows that $(x^{-n}|x^{m})_o^{d}$ is uniformly bounded for all $n,m \ge 1$. 
    Then \eqref{eq.defH_G} and \eqref{eq.gpineq} imply that there exists $C' >0$ (depending on $x$) such that
    \[
     \psi(o,x^{n+m}) \le \psi(o,x^n) + \psi(o,x^m) +C'
    \]
    for all $n,m \ge 1$.  We then see by Fekete's lemma that the limit in \eqref{eq.defSTLcalH} exists and equals
    \[
    \ell_\psi(x) = \inf_{k \ge 1} \frac{1}{k} (\psi(o,x^{k}) +C').
    \]

    For the equivalence statements we note that $(1)$ immediately implies $(2)$. We therefore just need to show that $(2)$ implies $(1)$. Suppose $(2)$ holds and note that by the definition of $\calH_\G$ there exists $C>0$ such that $|\psi(o,xg) - \psi(o,x)| \le Cd(o,g)+C$ for all $x,g \in \G$. Also, from \cite[Lemma 3.11]{cantrell-reyes.manhattan} there exist $R,D > 0$ such that the following holds. For any $x \in \G$ there is $y \in \G$ with $d(x,y) < R$ and $(y^{-n}|y^m)_o^d \le D$ for all $n,m \ge 1$. In particular, the same argument as before implies that $\ell_\psi(y)=\inf_{k \geq 1 }{\frac{1}{k}(\psi(o,y^k)+\lam D+\lam)}\leq \psi(o,y)+\lam D \lam$, and hence we obtain
   \[
        0 \le \ell_\psi(y) \le \psi(o,y) +\lam D \lam \le \psi(o,x) +CR+C+ \lam D+ \lam.
\]
This completes the proof since $CR+C+ \lam D+ \lam$ is independent of $x\in \G$.
\end{proof}

Extending the notions from \eqref{sec.metricnotions}, we can talk of two hyperbolic metric potentials $\p,\p_\ast$ being quasi-isometric (resp. roughly similar/isometric). That is, whether $\p,\p_\ast$ satisfy \eqref{eq.defqi} for $F$ being the identity map $(\G,\p) \ra (\G,\p_\ast)$. As a corollary of \Cref{lem:MLSnonegative} we get that $\p,\p_\ast \in \calH_\G$ are roughly isometric if and only if $\ell_\p=\ell_{\p_\ast}$.

\begin{definition}\label{def.H^+_G}
   Let $\calH_\G^+\subset \calH_\G$ be the set of hyperbolic metric potentials such that $\ell_\p\geq 0$. 
\end{definition}
Equivalently, by \Cref{lem:MLSnonegative} we have that $\calH_\G^+$ coincides with the set of hyperbolic metric potentials that are bounded below. The next lemma shows that potentials in $\calH_\G^+$ coincide with those in $\wh \calD_\G$ up to functions in $\calB_\G$.


\begin{lemma}\label{lem.GPpositiveH_G}
    If $\psi\in \calH_\G^{+}$ then there exists $C'>0$ such that $(x|y)^\psi_w\geq -C'$ for all $x,y,w\in \G$. In consequence, $\p$ equals $d-b$ for some $d\in \wh\calD_\G$ and $b\in \calB_\G \cap \calD_\G$.
\end{lemma}

Before proving this lemma we recall the notion of quasi-center. Given a function $\p:\G \times \G \ra \R$ and $\k\geq 0$, we say that $z\in \G$ is a $(\p,\k)$-\emph{quasicenter} for the points $x,y,w\in \G$ if 
   $$\max \{|(x|y)_{z}^\psi|,|(y|x)_{z}^\psi|,|(y|w)_{z}^\psi|,|(w|y)_{z}^\psi|,|(w|x)_{z}^\psi|,|(x|w)_{z}^\psi|\} \leq \k.$$
It is straightforward to check that such a quasicenter $z$ satisfies     \begin{equation}\label{eq.propQCenter}
|2(x|y)_w^\psi-\psi(z,w)-\psi(w,z)|\leq 6\k.
\end{equation}

For $d\in \calD_\G$ there exists $\k$ such that any triple in $\G$ has a $(d,\k)$-quasicenter. This follows from Gromov hyperbolicity and the rough geodesic property.

\begin{proof}
Suppose that $(\cdot|\cdot)^\psi_\cdot \leq \lam (\cdot|\cdot)^{d_0}_\cdot+\lam$ for some $d_0\in \calD_\G$ and $\lam>0$ and that $d_0$. We also let $\k$ such that $d_0$ has quasicenters. Given $x,y,w \in \G$ let $p$ be a $(\k,d_0)$-quasicenter. Then $p$ is $(\hat \k,\psi)$-quasicenter for $\hat \k=\lam \k+\lam$, and hence
\begin{align*}
    2(x|y)^\psi_w & =\psi(x,w)+\psi(w,y)-\psi(x,y) \\
       & =\psi(x,w)+\psi(w,y)-\psi(x,p)-\psi(p,y)-2(x|y)_p^\psi\\
       & =\psi(x,w) -\psi(x,p)+\psi(w,y)-\psi(p,y)-2(x|y)_p^\psi\\
       & =\psi(p,w)+\psi(w,p)-2(x|y)_p^\psi-2(x|w)_p^\psi-2(w|y)^\psi_p \\
       & \geq 2\inf \psi -6\hat\k=:-C'.
\end{align*}
Since $\psi \in \calH_\G^{+}$, we deduce that $(x|y)_w^\psi$ is bounded below by some constant independent of $x,y,w$, concluding the proof of the first assertion.

To prove the second assertion it is enough to check that $d:\G \times \G$ given by
\begin{equation}\label{eq.defdfrompsi}
d(x, y):=\begin{cases}
\p(x, y)+2C' & \text { if } x \neq y \\
0 & \text { otherwise }
\end{cases}
\end{equation}
 is $\G$-invariant, nonnegative and satisfies the triangle inequality (cf.~\cite[Lemma~4.4]{cantrell-reyes.manhattan}). This implies that $d \in \widehat \calD_\G$, and by construction $b:= d-\p$ belongs to $\calB_\G\cap \calD_\G$.
\end{proof}

The lemma above implies the following structural result for hyperbolic metric potentials.

\begin{proposition}\label{prop.descriptionH_G}
    $\calH_\G$ is generated as a vector space by $\widehat\calD_\G$. Indeed, for every $\psi\in \calH_\G$ there exist $d\in \widehat\calD_\G$, $d_\ast \in \calD_\G$ such that
    \[\psi=d_\ast-d.\]
\end{proposition}

\begin{proof}
    Let $\psi \in \calH_\G$ and consider $\hat d_\ast\in \calD_\G$ such that $\ell_\psi(x)\leq \ell_{\hat d_\ast}(x)$ for all $x\in \G$. Then $\hat d:= \hat d_\ast -\psi$ belongs to $\calH_\G^{+}$ and by \Cref{lem.GPpositiveH_G} we have $\hat d=d-b$ for $d\in \wh \calD_\G$ and $b\in \calB_\G \cap \calD_\G$. 
 This implies that $\p=d_\ast-d$ for $d_\ast:=\hat d_\ast+b\in \calD_\G$ and $d\in \wh \calD_\G$.
\end{proof}

The proposition above allows use to promote properties from hyperbolic distance-like functions in $\widehat \calD_\G$ to  hyperbolic metric potentials. For example, using \Cref{prop.descriptionH_G} and results from \cite[Section 3.1]{cantrell-reyes.manhattan} we can deduce the following.

\begin{corollary}\label{coro.propertiesH}
 Each  $\psi \in \calH_\G$ satisfies the following properties:
    \begin{enumerate}
        \item Left and right Lipschitz: for any $d \in \Dc_\G$ there exists $A >0$ such that $|\psi(x,y) - \psi(x',y')|  < Ad(x,x')+Ad(y,y')+A$ for all $x,x',y,y'\in \G$.
        \item For any $d \in \Dc_\G$ there exists a $D_0 > 0$ such that if $d(o,x) > D$ for $D > D_0$ and $(x|x^{-1})_o^d \le C$ for some $C>0$ then 
        \[
        |\psi(o,x^m) - \ell_\psi(x^m)| \le R_{C,D}  \text{ for all } m\geq 1,
        \]
        where $R_{C,D}$ is a constant depending on $C,D$.
    \end{enumerate}
\end{corollary}

Using \Cref{prop.descriptionH_G}, \eqref{eq.defH_G} and \eqref{eq.propQCenter}, we can also deduce the following.

\begin{lemma}\label{lem.symmetryGroProd}
    Let $\psi \in \calH_\G$ and $d \in \calD_\G$. Then for any $\k>0$ there exists $\hat\k$ such that $(d,\k)$-quasicenters are $(\psi,\hat\k)$-quasicenters. In particular, if $z$ is a $(d,\k)$-quasicenter for $x,y,w$ in $\G$ then
\begin{equation}\label{eq.symmetryGP}
|2(x|y)_w^\psi-\psi(z,w)-\psi(w,z)|\leq 6\hat\k.
    \end{equation}
\end{lemma}

Applying this lemma to $w=o$ in $\G$, we deduce that the Gromov product of hyperbolic metric potentials is almost symmetric.


\begin{corollary}\label{coro.symGP}
        For any $\psi \in \calH_\G$ there exists $R>0$ such that for all $x,y\in \G$:
\begin{enumerate}
    \item $|(x|y)_o^\psi-(y|x)_o^\psi|\leq R$; and,
    \item $|(x|y)_o^{\psi}-(x|y)_o^{\hat\psi}|\leq R$, for $\hat\psi(u,v)=\psi(v,u)$.
\end{enumerate}    
\end{corollary}

\begin{proof}
    Item (1) follows directly from \Cref{lem.symmetryGroProd} after comparing $(x|y)_o^\psi$ and $(y|x)_o^\psi$ with $\psi(z,o)+\psi(o,z)$ for $z$ a quasicenter for $x,y,o$ with respect to a pseudo metric in $\calD_\G$. Item (2) follows from (1) after noting that $(x|y)_o^{\hat\psi}=(y|x)_o^\psi$.
\end{proof}

As another corollary we obtain that hyperbolic metric potentials in $\calH_\G^+$ are Gromov hyperbolic.

\begin{corollary}
For any $\psi \in \calH_\G^{+}$ there exists $\d > 0$ such that
    \[
    (x|y)_w^\psi \ge \min \{ (x|z)_w^\psi, (z|y)_w^\psi\} - \d
    \]
    for all $x,y,z,w\in \G$.
\end{corollary}

\begin{proof}
 By \Cref{lem.GPpositiveH_G} it is enough to assume that $\psi \in \widehat\calD_\G$. Also, from \Cref{coro.symGP} there exists $R>0$ such that 
 $|(x|y)^\psi_o-(x|y)_o^{\hat\psi}|\leq R$ for all $x,y\in \G$, where $\hat\psi(x,y)=\psi(y,x)$. Therefore it is enough to prove the result for $d=\psi+\hat\psi$, which belongs to $\ov\calD_\G$. But $d$ is a symmetric pseudo metric on $\G$, which we know is Gromov hyperbolic by \cite[Lemma 6.3]{cantrell-reyes.manhattan}.
\end{proof}

We are ready to prove \Cref{prop.H_G=temp}.

\begin{proof}[Proof of \Cref{prop.H_G=temp}]


We already observed that tempered potentials belong to $\calH_\G$. To prove the converse, let $\psi\in \calH_\G$, for which we expect properties \eqref{Eq:RG} and \eqref{Eq:QE} to hold. We note that \eqref{Eq:RG} easily follows from \eqref{eq.defH_G}. To prove \eqref{Eq:QE} we first assume that $\psi \in \calH_\G^+$. By \Cref{lem.GPpositiveH_G} we can assume that $\psi\in \widehat\calD_\G$, and by \Cref{coro.symGP} we can assume that $\psi \in \ov\calD_\G$. But then $\psi$ is the difference of two pseudo metrics in $\calD_\G$ by \cite[Corollary 5.4]{cantrell-reyes.manhattan}, and hence $\psi$ clearly satisfies \eqref{Eq:QE} because pseudo metrics in $\calD_\G$ do. The general case then follows by linearity and \Cref{prop.descriptionH_G}.
\end{proof}

\Cref{prop.H_G=temp} allows us to talk about the Busemann function associated to $\psi \in \calH_\G$ as in \eqref{eq.defbusemann}. It also allows us to apply the results from \cite[Section 2.6]{cantrell-tanaka.manhattan}. 


\subsection{Extending the space of metric structures}\label{sec.HMPmetricstructures}

We want to compare hyperbolic metric potentials similarly to pseudometrics in $\calD_\G$. If $\p_\ast \in \calH_\G$ and $\p\in \calH_\G^+$, the \emph{dilation} of $\p_\ast$ with respect to $\p$ is 
\[\Dil(\p_\ast,\p):=\inf\{\lam:\ell_{\psi_\ast}\leq \lam\ell_\p\} \in \R \cup \{\infty\}, \]
where we define $\Dil(\p_\ast,\p)=\infty$ if no such $\lam$ exists. Since $\p_\ast$ is not necessarily positive, it can happen that $\Dil(\p_\ast,\p)$ is negative. However, the dilation is always finite if $\p\in \calD_\G$. Finiteness of the dilations holds more generally for the following class of functions.

\begin{definition}\label{def.H_G^++}
Let $\calH_\G^{++}\subset \calH_\G$ be the set of all hyperbolic metric potentials $\p$ such that $\ell_\p \geq \ell_d$ for some $d\in \calD_\G$.
\end{definition}

In the introduction, $\calH_\G^{++}$ was referred to the set of hyperbolic metric potentials $\p$ that are proper. That is, those $\p$ such that $\{x\in \G: \p(o,x)<T\}$ is finite for any $T$. By \Cref{lem:MLSnonegative}, functions satisfying $\ell_\p\geq \ell_d$ for some $d\in \calD_\G$ are proper. We can use \Cref{lem.GPpositiveH_G} to prove the converse. Indeed, since properness is preserved under rough isometry we can suppose that $\p\in \wh\calD_\G$. For these functions, using \cite[Lemma 3.2]{cantrell-reyes.manhattan} and \cite[Lemma 2.5]{cantrell-reyes.approx} we have that for $m$ large enough, $\p$ is quasi-isometric to the (possibly asymmetric) word metric with generating set $S_m=\{x\in \G: \p(o,x)\leq m\}$. Therefore, $\p$ is proper if and only if $S_m$ is finite, which is true if and only $\p$ is quasi-isometric to a pseudo metric in $\calD_\G$ (and hence satisfies $\ell_\p \geq \ell_d$ for some $d\in \calD_\G$). 

We can also characterize hyperbolic metric potentials in $\calH_\G^{++}$ using the \emph{exponential growth rate}. As for pseudo metrics in $\calD_\G$, the exponential growth rate of $\p\in \calH_\G^+$ is the (possibly infinite) number
\begin{equation}\label{eq.defEGRtempered}
    v_\p = \limsup_{T\to\infty} \frac{1}{T} \log \#\{x \in \G: \p(o,x) < T\}.
\end{equation}
Note that $\p$ and $\p_\ast$ being roughly isometric implies $v_\p=v_{\p_\ast}$. By subadditivity, it is not hard to check that $v_\p$ is finite if and only if $\p$ is proper. We summarize this discussion in the next lemma.

\begin{lemma}\label{lem.charH_G^++}
Given $\p\in \calH_\G^+$ the following are equivalent:
\begin{enumerate}
\item $v_\p$ is finite;
\item $\p$ is proper;
\item $\p$ is quasi-isometric to a metric in $\calD_\G$; and,
\item $\p\in \calH_\G^{++}$.
\end{enumerate}
\end{lemma}


If we only consider rough similarity classes in $\calH_\G$, we obtain the following extensions of $\scrD_\G.$

\begin{definition}\label{def.metricstructuresH_G}
    We let $\scrH_\G$ denote the space of rough similarity equivalence classes in $\calH_\G$. Similarly, we let $\scrH^+_\G$ (resp. $\scrH_\G^{++}$) be the subsets of $\scrD_\G$ induced by $\calH_\G^+$ (resp. $\calH_\G^{++}$). 
\end{definition}

If $\psi\in \calH_\G$, we let $[\psi]$ denote its rough similarity class, which we also call a \emph{metric structure}. Since dilations for hyperbolic metric potentials in $\calH_G^{++}$ are finite, the formula \eqref{eq.defDel} also makes sense for metrics in this space, and hence we get a \emph{Thurston metric} $\Del$ on $\scrH_\G^{++}$ extending that on $\scrD_\G$. By \Cref{lem:MLSnonegative}, $\Del$ is a genuine metric.


\subsection{Extending the Bowen-Margulis construction}\label{sec.BMHMP}

Similar to pseudo metrics in $\calD_\G$, hyperbolic metric potentials in $\calH_\G^{++}$ admit well-behaved measures at infinity.

\begin{definition}\label{def.QCH_G}
    Given $\psi \in \calH_\G^{++}$, we say that a Borel probability measure $\nu$ on $\partial \G$ is \emph{quasi-conformal} for $\p$ if there exists some constant $C>1$ such that for every $x\in \G$ and $\nu$-almost every $\x\in \partial \G$ we have
\begin{equation}\label{eqqcmeasdefH_G}
     C^{-1}e^{-v_\p\b_o^\p(x,\x)}\leq \frac{dx_\ast\nu}{d\nu}(\x)\leq Ce^{-v_\p\beta_o^\p(x,\x)}.
\end{equation}
\end{definition}

\begin{proposition}\label{prop.existenceQCH_G}
    Any $\p\in \calH_\G^{++}$ 
    admits a quasi-conformal measure $\nu$ and any two such quasi-conformal measures are absolutely continuous with respect to each other.  
    Moreover, $\nu$ satisfies:
\begin{itemize}
    \item there exist $C >1$ and $R>0$ such that
        \[
        C^{-1} e^{-v_\p\p(o,x)} \le \nu(O(x,R)) \le         Ce^{-v_\p\p(o,x)} 
        \]
        for all $x \in \G$;
    \item $\nu$ is doubling with respect to any visual metric on $\partial \G$; and,
    \item $\nu$ is ergodic with respect to the action of $\G$ on $\partial \G$.
    \end{itemize}
\end{proposition}

\begin{proof}
First, we recall that $\p$ is a ($\G$-invariant) tempered potential by \Cref{prop.H_G=temp}. 
Next, we fix $d\in \calD_\G$ and consider the series 
\[s \mapsto \sum_{x\in \G}{\exp(-v_\p\p(o,x)-sd(o,x))}.\]
Since $v_\p \p$ has exponential growth rate 1 and $\p,d$ are quasi-isometric, we can check that the critical exponent of this series is $\thet=0$.  In particular, \cite[Proposition 2.7]{cantrell-tanaka.manhattan} applies and implies the existence of $\nu$ quasi-conformal for $\p$. The rest of the properties for $\nu$ then follow from \cite[Lemma 2.9]{cantrell-tanaka.manhattan}.
\end{proof}

This proposition and the classical arugment due to Patterson and Sullivan then implies the following.

\begin{corollary} \label{cor.exp}
$\psi \in \calH_\G^{++}$ there exists $R_0 > 0$ such that for any $R > R_0$ there are constants $C_1,C_2 > 0$ (depending on $R$) such that
\[
C_1 e^{v_\psi T} \le \#\{x \in \G: \psi(o,x) < T \} \le C_2 e^{v_\psi T}
\]
for all $T > 0$.
\end{corollary}

As for pseudo metrics in $\calD_\G$ we can use quasi-conformal measures to construct geodesic currents associated to functions in $\calH_\G^{++}$. Let $\p\in \calH_\G^{+}$ and consider $\hat\p\in \calH_\G^{+}$ given by $\hat\p(x,y)=\p(y,x)$.  Note that $\hat\p$ also belongs to $\calH_\G^{+}$ and that $v_{\hat \p}=v_\p$. 

Suppose now that $\p\in \calH_\G^{++}$ and let $\nu,\hat \nu$ be quasi-conformal measures for $\p$ and $\hat \p$ respectively. Applying \Cref{prop.H_G=temp}, \cite[Proposition 2.8]{cantrell-tanaka.invariant} and the discussion after the proof of that proposition, we deduce the existence of a geodesic current $\Lam_\psi \in \calC_\G$ in the measure class of $\hat \nu \otimes \nu$. Moreover, it satisfies
\begin{equation}\label{eq.defBMH_G}
  d\Lam_\psi(\x,\y)=\al(\xi,\y)\exp(2 v_\p(\x|\y)^\p_o)d\hat\nu (\x)d\nu(\y)  
\end{equation}
for a measurable function $\al$ that is essentially bounded and bounded away from zero.
We call any such $\L_\p$ a \emph{Bowen-Margulis} current for $\p$.
Note that if $[\p]\neq [\hat \p]$ then $[\Lam_\p] \neq [\Lam_{\hat\p}]$ in $\bbP\calC_\G$. In general, we always have $[\Lam_{\hat \p}]=[\iota_\ast \Lam_\p]$ for $\iota$ the flip on $\partial^2 \G$. In this setting, Bowen-Margulis currents are also ergodic.



Given a function $\ell: \G \ra \R$ satisfying $\ell(x^n)=n\ell(x)$ for all $x\in \G$ and $n\in \N$, we can define a function on positive linear combinations of rational currents in $\calC_\G$ as the linear extension of $\ell(\mu_x):=\ell(x)$. If $\ell=\ell_\psi$ for some $\psi\in \calH_\G$, then $\ell$ admits a (necessarily unique) continuous extension to $\calC_\G$. This is the content of the next theorem by Kapovich and Mart\'inez-Granado \cite{kapovich.martinezgranado}. For $\G$ torsion-free and $\psi \in \ov\calD_\G$, this was proved in \cite[Theorem 1.7]{cantrell-reyes.manhattan}.

\begin{theorem}\label{thm.continuousextension}
Given $\psi\in \calH_\G$, the stable translation length $\ell_\p$ continuously extends to $\calC_\G$.
\end{theorem}

As a corollary, we obtain a characterization for hyperbolic metric potentials in $\calH_\G^+$ and  $\calH_\G^{++}$ in terms of geodesic currents.

\begin{corollary}\label{coro.charH+currents}
    If $\p\in \calH_\G$ then we have:
    \begin{enumerate}
        \item $\p\in \calH_\G^+$ if and only if $\ell_\p(\mu)\geq 0$ for every $\mu\in \calC_\G$; and,
        \item $\p\in \calH_\G^{++}$ if and only if $\ell_\p(\mu)>0$ for every non-zero current $\mu\in \calC_\G$.  
    \end{enumerate}
\end{corollary}

\begin{proof}
    Item (1) is immediate from \Cref{lem:MLSnonegative} and the density of real multiples of rational currents in $\calC_\G$. 

    For Item (2), we first note that $\ell_d(\mu)>0$ for all $d\in \calD_\G$ and $\mu\in \calC_\G$, and hence the same property holds for $\ell_\p$ for $\p\in \calH_\G^{++}$ by \Cref{lem.charH_G^++}. Conversely, $\p \in \calH_\G$, $\p\in \calH_\G$ and $\ell_\p(\mu)>0$ for all non-zero currents $\mu$, then the function on $\bbP\calC_\G$ given by $[\mu]\mapsto \ell_\p(\mu)/\ell_d(\mu)$ is well-defined and continuous by \Cref{thm.continuousextension}, and positive. Since $\bbP\calC_\G$, we can find $c>0$ such that $\ell_\p \geq c\ell_d$, and hence $\p\in \calH_\G^{++}$ by \Cref{lem:MLSnonegative} and \Cref{lem.charH_G^++}. 
\end{proof}



\subsection{Examples}\label{sec.exHMP}

It turns out that there are many natural functions that belong to $\calH_\G$. Similar to metrics in $\calD_\G$ induced by geometric actions of $\G$ on geodesic spaces, we obtain metrics in $\ov\calD_\G$ (and hence in $\calH_\G^+)$ from isometric actions that are not necessarily proper. In \cite[Theorem~1.6]{cantrell-reyes.manhattan} this was proven for the following actions:
\begin{enumerate}
    \item Actions on coned-off Cayley graphs for finite, symmetric generating sets, where we cone-off a finite number of quasi-convex subgroups of infinite index.
    \item Non-trivial Bass-Serre tree actions with quasi-convex vertex stabilizers of infinite index.
    \item Non-trivial small actions on $\R$-trees, when $\G$ is a surface group or a free group.
\end{enumerate}

All these examples yield symmetric hyperbolic metric potentials. Non-symmetric hyperbolic metric potentials include word metrics for non-symmetric finite generating sets. More non-symmetric hyperbolic metric potentials can be constructed from Anosov representations.

A representation $\rho:\G \ra \SL_d(\R)$ is $k$-\emph{dominated} ($1\leq k \leq d-1$) if there exist $C,\lam>0$ and a finite generating subset $S\subset \G$ such that
\begin{equation*}\label{eq.defdominated}
    \frac{\sig_k(\rho(x))}{\sig_{k+1}(\rho(x))}\geq Ce^{\lam|x|_S} \text{ for all }x\in \G.
\end{equation*}

See e.g. \cite{BPS}. If $\rho$ is $1$-dominated, then the formula
\[d_{\rho,1}(x,y):= \log \sig_1(\rho(x^{-1}y))\]
defines an element of $\hat\calD_\G \subset \calH_\G$ \cite[Lemma~7.1]{cantrell-tanaka.invariant}. If $\G$ is $k$-dominated for all $k$, then $\G$ is called \emph{Borel Anosov}. In this case, the functions
\[d_{\rho,k}(x,y):= \log \sig_k(\rho(x^{-1}y))\] belong to $\calH_\G^{++}$ for all $k$. Indeed, we have $d_{\rho,k}=d_{\wedge^k \rho,1}-d_{\wedge^{k-1}\rho,1}$ for all $k\geq 2$, where $\wedge^r \rho$ denotes the $r$th exterior product representation induced by $\rho$. Since $\rho$ is $k$-dominated if and only if $\wedge^k \rho$ is 1-dominated, we deduce that $d_{\rho,\k}\in \calH_\G$.

So far, all the hyperbolic metric potentials we have considered belong to $\calH_\G^+$. Examples of potentials unbounded by above and below are given by non-trivial quasimorphisms. A \emph{quasimorphism} is a function $\varphi:\G \ra \R$ such that 
\begin{equation}\label{eq.defQM}\sup_{x,y\in \G}{|\varphi(x)+\varphi(y)-\varphi(xy)|} <\infty.
\end{equation}
If $\varphi$ is a quasimorphism and we set $\psi(x,y)=\varphi(x^{-1}y)$, then $\psi$ is $\G$-invariant and \eqref{eq.defQM} translates to $\sup_{x,y,w} |(x|y)_w^\psi|<\infty$. In particular, $\psi \in \calH_\G$. If $\varphi$ is unbounded, then $\ell_\psi$ is non-trivial, and in this case we have $\ell_\p(x^{-1})=-\ell_\p(x)$. In particular, $\ell_\p$ is unbounded by above and below. For non-elementary hyperbolic groups the space of quasi-morphisms is infinite dimensional \cite{epstein-fujiwara}.

\section{Construction of the joint translation spectrum}\label{sec.constructionJTS}

In this section we prove the existence of the joint translation spectrum and prove Theorems \ref{thm.intro.cone} and \ref{thm.intro.jtspectrum}. Throughout this section we fix $\p_1,\dots,\p_n \in \calH_\G$ and $\p\in \calH_\G^{++}$ and let $\D, \D_\p, \bfL, \bfL_\p$ be as in \Cref{sec.intro.TC+JTS}. We also fix an auxiliary metric $d\in \calD_\G$. We denote by $\|\cdot\|$ the Euclidean norm on $\R^n$. 

\subsection{Preliminary results}
We begin by proving some preliminary results for functions in $\calH_\G$. The next lemma is a direct consequence of \Cref{coro.propertiesH}(2).

\begin{lemma}\label{lem.closeto}
For any $\ep > 0$ there exists $C>0$ such that if $x \in \G$ satisfies $(x|x^{-1})^d_{o} \le \ep$ then $\|\bfL(x^m) - \D(x^m)\| \le C$ for all $m\geq 1$.
\end{lemma}
To compare $\bfL$ and $\D$ we therefore need to understand when $x\in  \G$ has a small Gromov product with its inverse with respect to $d$. This leads us to the following definition.
\begin{definition}
We say that a group element $x \in \G$ is $(d,\epsilon)$-\emph{proximal} for $\epsilon \ge 0$ if $(x|x^{-1})^d_{o} \le \epsilon$.
\end{definition}
The following technical lemma will be key in the proofs of Theorems \ref{thm.intro.cone} and \ref{thm.intro.jtspectrum} . It will allow us to find appropriate proximal elements which in turn will enable us to control the displacement and translation vectors of certain products.

\begin{lemma}\label{lem.construct}
    There exists a finite set $F \subset \G$ and $\epsilon >0$ such that given $x,y \in \G$ there are $f_1, \ldots, f_6 \in F$ such that:
    \begin{enumerate}
        \item the elements $(f_1 x f_2)^{m_1}, (f_4yf_5)^{m_2}$ and $(f_1xf_2)^{m_1}f_3(f_4yf_5)^{m_2}f_6$ are $(d,\epsilon)$-proximal for all $m_1,m_2 \geq 1$;
        \item $d(o,(f_1xf_2)^{m_1}f_3(f_4yf_5)^{m_2}f_6) = d(o,(f_1xf_2)^{m_1})+ d(o,(f_4yf_5)^{m_2})+O(1)$ for all $m_1,m_2 \ge 1$, where the implied error is independent of $x,y$ and $m_1,m_2$; and, 
        \item $\D((f_1xf_2)^{m_1}f_3(f_4yf_5)^{m_2}f_6) = \D((f_1xf_2)^{m_1}) + \D((f_4yf_5)^{m_2}) +O(1)$ for all $m_1,m_2 \ge 1$, where the implied error is independent of $x,y$ and $m_1,m_2$.
    \end{enumerate}
\end{lemma}

The finite set $F$ appearing in this key lemma is very similar to Schottky sets appearing in the recent works \cite{BMSS,gouezel.schottky,choi.random1} on random walks on Gromov hyperbolic spaces. It traces back to Abels--Margulis--Soifer \cite{AMS} in the matrix setting which was used in \cite{breuillard-sert} to study the matrix version of the joint translation spectrum. The precise version we will use is due to \cite{cantrell.mixing}.

\begin{proof}
Fix a word metric $d_S$ on $\G$ associated to a finite symmetric generating set $S$. To prove the lemma it is convenient to introduce a combinatorial tool for studying hyperbolic groups, the \emph{Cannon coding} \cite{cannon}. This is a finite directed graph that encodes the pair $(\G,S)$. More precisely the Cannon graph $(\calG, V,E, v_\ast)$ is a finite directed graph $\calG$ with vertex set $V$, directed edge set $E$ and a distinguished vertex $v_\ast$, with the following properties:
\begin{enumerate}
    \item there is a labelling $l : E \to S$ associating a generator in $S$ to each directed edge in $E$;
    \item let $P_\ast$ denote the finite directed paths starting at $v_\ast$. Then, the map $\textnormal{ev} : P_\ast \to \G$ that maps the path following the edges $e_1, e_2,\ldots, e_k$ (where $e_1$ starts at $v_\ast$) to $l(e_1) l(e_2) \cdots l(e_k) \in \G$ is a bijection; and,
    \item the bijection in $(2)$ maps paths of length $k$ (i.e. paths that consist of $k$ edges) to group elements in $\G$ of $S$-word length $k$.
\end{enumerate}
In particular it follows that finite paths in $\calG$ correspond to geodesics in the word metric for $S$.

We will use this combinatorial structure to prove the lemma, so take arbitrary $x,y \in \G$. Then by \cite[Lemma~2.16]{cantrell.mixing} there exists a finite set $F \subset \G$ independent of $x,y$ such that there are $f_1,f_2, f_4,f_5\in F$ such that $f_1xf_2$ and $f_4yf_5$ are obtained by multiplying the labellings in a closed loop in the Cannon graph. It follows that $(f_1 x f_2)^{m_1}$ and $(f_4 y f_5)^{m_2}$ are $(d_S, 0)$-proximal for all $m_1,m_2 \ge 1$. We can also enlarge $F$ by adding a finite number of group elements (corresponding to finite paths in $\calG$) so that there are $f_3,f_6 \in F$ so that $(f_1xf_2)^{m_1}f_3(f_4yf_5)^{m_2}f_6$ is represented by a loop in the Cannon graph for any $m_1,m_2 \ge 1$. This concludes the proof of $(1)$ since $(d_S,0)$-proximal implies $(d,\ep)$-proximal for some $\ep$ depending only on $d$ and $d_S$.

Parts $(2)$ and $(3)$ also easily follow. Indeed, we know that $(f_1xf_2)^{m_1}f_3(f_4yf_5)^{m_2}f_6$ is a $d_S$-geodesic and so the conclusion follows from the Morse Lemma (and the definition of $\calH_\G$) when comparing $d_S$ and $d, \p_1,\ldots,\p_n$.
\end{proof}

\subsection{Proof of results} Now we prove the results from \Cref{sec.intro.TC+JTS}. For a subset $A\subset \R^n$ we let $\textnormal{CH}(A)$ denotes its convex hull. We begin with the proof of \Cref{thm.intro.cone}.

\begin{proof}[Proof of \Cref{thm.intro.cone}]
To show convexity of $\calT\calC(\D)$ it is enough to show that $m_1\bfL(x)+m_2\bfL(y)$ belongs to $\calT\calC(\D)$ for all $m_1,m_2 \geq 1$ and $x,y\in \G$. To do this, we apply \Cref{lem.closeto} and \Cref{lem.construct} to find a constant $C$ (independent of $x,y$) and $f_1,\dots,f_6 \in \G$ such that 
\[\|\bfL((f_1xf_2)^{m_1}f_3(f_4yf_6)^{m_2}f_6)-m_1\bfL(x)-m_2\bfL(y)\| \leq (m_1+m_2)C\] 
Applying this to powers $x^k$ and $y^k$, then we get the result after dividing by $k$ and letting $k$ go to infinity.

Next, we prove that $\calT\calC(\D)$ is the asymptotic cone of $\D(\G)$. It $\vx\in \R^n$ is the limit of $\al_k \D(x_k)$ for some $x_k\in \G$ and $\al_k \to 0$, then again by \Cref{lem.closeto} and \Cref{lem.construct} we can find $C>0$ and elements $x_k'\in \G$ such that $\|\D(x_k)-\bfL(x_k')\|\leq C$ for all $k$. Then $\vx=\lim_{k}{\al_k \bfL(x_k')}$ and hence belongs to $\calT\calC(\D)$. The other inclusion is easy, since asymptotic cones are closed and any vector $\al \bfL(x)$ equals the limit of $\frac{\al}{k}\D(x^k)$ as $k$ tends to infinity.

Finally, we show that $\calT\calC(\D)$ has non-empty interior if and only if $\p_1,\dots,\p_n$ are independent. Clearly if $\p_1,\dots,\p_n$ are not independent then $\calT\calC(\D)$ is contained in a codimension-1 subspace of $\R^n$ and has empty interior. Conversely, suppose that $\p_1,\dots,\p_n$ are independent. Then we can find non-torsion elements $x_1,\dots,\x_n \in \G$ such that the vectors $\bfL(x_1),\dots, \bfL(x_n)\in \R^n$ are affinely independent. Convexity of $\calT\calC(\D)$ then implies that \[
    \textnormal{CH}(\{\bfL(x_1), \ldots, \bfL(x_{n})\}) \subset \calT\calC(\D)
    \]
    and it follows that $\calT\calC(\D)$ has non-empty interior in $\R^n$.
\end{proof}

We continue with the proof of \Cref{thm.intro.jtspectrum}. We need the following lemma that provides a useful characterization of Hausdorff convergence (see \cite[Section 4.4]{ambrosio.book}).
\begin{lemma} \label{lem.hausconv}
    Suppose that $(X,d_X)$ is a compact metric space. A sequence $K_T$, $T>0$ of compact subsets of $X$ converges in the Hausdorff metric to a compact subset of $X$ if and only if for every $\d > 0$ there exists $T_0>0$ such that for all $T_1 \ge T_0$ and for every $x \in K_{T_1}$ we have that $\limsup_{T\to\infty} d(x,K_T) \le \d$.
\end{lemma}

\begin{proof}[Proof of \Cref{thm.intro.jtspectrum}]
\textbf{Existence of the limits.}
We begin by showing that the limits stated in the theorem exist.
   Since $\p\in \calH_\G^{++}$, by \Cref{lem.GPpositiveH_G} we can assume without loss of generality that $\p$ satisfies the triangle inequality. Let $R$ be large enough so that $S_R(T)$ grows exponentially quickly as in 
   \Cref{cor.exp}. Since $\p\in \calH_\G^{++}$, the sets $K_T=\D_\p(S_R(T))$ are uniformly bounded in $\R^n$, and so by \Cref{lem.hausconv} we want to show that for all $\d > 0$ there is $T_0>0$ such that for $ T_1 \ge T_0$ and all $\vx \in \D_\p(S_R(T_1))$ we have that $\limsup_{T\to\infty} \|\vx- K_T \| \le \d$.
   
    To that end fix $x \in S_R(T_1)$ and let $F$ denote the finite set from \Cref{lem.construct}. Then by \Cref{lem.closeto} and \Cref{lem.construct} there are $\ep, C>0$ such that we can find $f_1, f_2 \in F$ such that $(f_1xf_2)$ is $(d,\ep)$-proximal and
    \[
    \| \bfL((f_1 x f_2)^m) - \D((f_1xf_2)^m)\| \le C \ \text{ for all $m\ge 1$.} 
    \]
    Note that the constant $C >0$ is independent of $x \in \G$.
    We also have that $\bfL((f_1xf_2)^m) = m \bfL(f_1xf_2)$, $\|\D(f_1xf_2) - \D(x)\| \le M$ and $\|\D_\p(x)\| \le \widetilde{C}$ where $M, \widetilde{C} >0$ are constants independent of $x \in \G$  and $m \ge 1$ by \Cref{coro.propertiesH}. Hence we have that
    \[
    \left\| \bfL(f_1xf_2) - \frac{1}{m}\D((f_1xf_2)^m)\right\| \le \frac{C}{m} 
    \]
    and by the triangle inequality
    \begin{equation}\label{eq.tineq}
    \left\|\D(x) - \frac{1}{m} \D((f_1xf_2)^m)\right\| \le C\left( 1 + \frac{1}{m} \right) + M.        
    \end{equation}

    Now for each $ T \ge T_1$ we can find $R'>0$ (independent of $T$ but possibly depending on $T_1$), $m_T \in \mathbb{Z}$ and $r_T \in S_R(R')$  such that $x_T = r_T(f_1xf_2)^{m_T}$ belongs to $S_R(T)$ (we must enlarge $R$ if necessary, but independently on $T$ or $T_1$). By $(d,\ep)$-proximality of $f_1xf_2$, notice that there is $C' >0$ independent of $ T$ and $T_1$ such that
    \[
    |\p(o,x_T) - m_T \, \p(o,x)| \le C' m_T
    \]
    for all $T > T_1$ large enough (note that the $m_T$'s tend to infinity). We then have by the triangle inequality
    \begin{align*}
    \left\|\D_\p(x) -  \D_\p(x_T)\right\| &\le \left\|\D_\p(x) - \frac{1}{m_T \p(o,x)}\D((f_1xf_2)^{m_T})\right\|\\
    &\hspace{1cm} + \left| \frac{1}{m_T\p(o,x)} - \frac{1}{\p(o,x_T)}\right|\left\|\D((f_1xf_2)^{m_T}\right\|\\
    &\hspace{2cm}+ \frac{1}{\p(o,x_T)} \left\|\D((f_1xf_2)^{m_T}) - \D(r_T(f_1xf_2)^{m_T})\right\|.
    \end{align*}
    Using (\ref{eq.tineq}) we see that
    \[
    \left\|\D_\p(x) - \frac{1}{m_T\p(o,x)}\D((f_1xf_2)^{m_T})\right\| \le \frac{C}{T_1- R} \left( 1 + \frac{1}{m_T}\right) + \frac{M}{T_1-R},
    \]
    and we also have that
    \begin{align*}
         \left| \frac{1}{m_T\p(o,x)} - \frac{1}{\p(o,x_T)}\right|\left\|\D((f_1xf_2)^{m_T})\right\| 
    & \le \left| \frac{C'm_T}{m_T\p(o,x)\p(o,x_T)}\right|\widetilde{C}\p((f_1xf_2)^{m_T})\\
    & \le \left( \frac{\widetilde{C}C'}{T_1-R}\right) \left( \frac{T+2R+R'}{T-R}\right).   
    \end{align*}

   It is also easy to see by \Cref{lem.closeto} that there is $L>0$ independent of $x, T$ (but depending on $T_1$) such that 
    \[
    \frac{1}{\p(o,x_T)} \| \D((f_1xf_2)^{m_T}) - \D(r_T(f_1xf_2)^{m_T})\| \le  \frac{L}{T-R}.
    \]
    Combining these estimates with our use of the triangle inequality above and taking the limit as $T \to \infty$ we see that
    \[
    \limsup_{T\to\infty} \left\|\D_\p(x) -  \D_\p(x_T)\right\| \le \frac{C''}{T_1-R}
    \]
    for some $C'' >0$ that crucially is independent of $T_1$ and $x \in S_R(T_1)$.
    This proves the convergence of $
    \D_\p(S_R(T))$ for each $R$ sufficiently large, and it is easy to see that this limit is independent of $R$. Indeed, for $R >0$ sufficiently large and for any $R_1 >R$ there exists $C >0$ depending on $R$ such that the Hausdorff distance between $\D_\p(S_R(T))$ and $\D_\p(S_{R_1}(T))$ is at most $C(T-R_1)^{-1}$ for all $T$ large enough. This concludes the proof of the convergence for $\D_\p(S_R(T))$.

    To prove the convergence for $\bfL_\p(S_R^{\ell_\p}(T))$ first note that
    by \cite[Lemma 3.11]{cantrell-reyes.manhattan} and \Cref{coro.propertiesH} there exists $M>0$ such that the following holds. For any $x\in \G$ we can find $y \in \G$ with $\p(y,x) < M$ and both $|\ell_\p(y) -\p(o,x)| < M$, $\|\bfL(y) - \D(y)\| < M$. In particular, there exists $R_0 > R$ and $C>0$ independent of $x \in \G$ such that for any $x \in S_R(T)$ we can find $y \in S_{R_0}^{\ell_\p}(T)$ with $\| \bfL_\p(y) - \D_\p(x)\| \le C T^{-1}$. This implies that 
    \[
    \D_\p(S_R(T)) \subset \{ \vx \in \R^n : \text{ there exists $y \in S_{R_0}^{\ell_\p}(T)$ with } \| \bfL_\p(y) - \vx\| \le C T^{-1} \}
    \]

     Conversely, it is easy to see that given $R_0$, there exist $M' >0$ and $R_0$ such that given $y \in S_{R_0}^{\ell_\p}(T)$ there exists $x \in S_{R_1}(T)$ such that $\p(y,x)<M'$ and $\|\bfL(y) - \D(x)\| < M'$. We therefore see that there is $C' > 0$ such that
\[
  \bfL_\p(S_{R_0}^{\ell_\p}(T)) \subset \{ \vx \in \R^n : \text{ there exists $y \in S_{R_1}(T)$ with } \| \D_\p(y) - \vx\| \le C' T^{-1} \}.
\]
The required Hausdorff convergence then follows easily from these two set inclusions and from the convergence for $\D(S_R(T))$.

\textbf{Convexity.} We now show that the joint translation spectrum is convex, and the argument is similar to that for \Cref{thm.intro.cone}. Since $\J_\p(\D)$ is closed it is enough to show that for $\mathbf{r},\mathbf{s} \in \J_{\p}(\D)$ we have $\frac{1}{2}(\mathbf{r}+\mathbf{s}) \in \J_{\p}(\D).$ To do this we first find sequences $x_k, y_k \in \G$ and $T_k \in \Z_{\ge 0}$ such that $x_k, y_k \in S_R(T_k)$, $T_k \to \infty$ as $k\to\infty$ and
    \[
    \D_\p(x_k) \to \mathbf{r} \text{ and } \D_\p(y_k) \to \mathbf{s}
     \]
     as $k\to\infty$.
     Now by \Cref{lem.closeto} and \Cref{lem.construct} there is a finite set $F \subset \G$ and sequences $f_{1,k}, \ldots, f_{6,k} \in F$ for $k\ge 1$ such that for $\widetilde{x}_k = f_{1,k}x_k f_{2,k}$ and $\widetilde{y}_k = f_{4,k}y_k f_{5,k}$ we have that
     \[
\bfL(\widetilde{x}_k f_{3,k} \widetilde{y}_k f_{6,k}) = \bfL(\widetilde{x}_k) + \bfL(\widetilde{y}_k) + O(1), \  \ell_{\p}(\widetilde{x}_k f_{3,k} \widetilde{y}_k f_{6,k}) = \p(o,\widetilde{x}_k) + \p(o,\widetilde{y}_k) + O(1)
\]
and 
\[
\bfL(\widetilde{x}_k) = \D(\widetilde{x}_k) + O(1), \ \bfL(\widetilde{y}_k) = \D(\widetilde{y}_k) +O(1),
     \]
     where the implied errors are independent of $k$.
     We then have that
     \[
     \frac{\bfL(\widetilde{x}_k f_{3,k} \widetilde{y}_k f_{6,k})}{\ell_{\p}(\widetilde{x}_k f_{3,k} \widetilde{y}_k f_{6,k}) } \to \frac{\mathbf{r}}{2} + \frac{\mathbf{s}}{2}
     \]
     as $k\to\infty$ and the proof is complete.

\textbf{Non-empty interior.}
We now show that $\J_\p(\D)$ has non-empty interior if and only if the associated potentials are independent, exactly as in \Cref{thm.intro.cone}. If $\p_1,\dots,\p_n,\p$ are dependent then there exist $c_1, \ldots, c_n, c_{n+1}$, not all simultaneously $0$, such that
    \[
    \sum_{i=1}^n c_i\ell_{\p_i}(x) + c_{n+1}\ell_{\p}(x) = 0
    \]
    for all $x\in\G$. We then deduce that $\J_{\p}(\D)$ is contained in the hyperplane with normal vector $(c_1,\ldots, c_n)$ and so has empty interior.

    Conversely suppose that $\p_1,\ldots, \p_n, \p$ are independent. Then we can find non-torsion elements $x_1,\ldots,x_{n+1}\in \G$ such that the vectors $\bfL_\p(x_1),\ldots, \bfL_\p(x_{n+1}) \in \R^{n+1}$ are affinely independent. Then by the above, $\J_\p(\D)$ is convex and so 
    \[
    \textnormal{CH}(\{\bfL_\p(x_1), \ldots, \bfL_\p(x_{n+1})\}) \subset \J_{\p}(\p_1,\ldots, \p_n)
    \]
    and it follows that $\J_{\p}(\D)$ does not have empty interior in $\R^n$.
\end{proof} 

We record the following useful corollary that follows easily from \Cref{thm.intro.jtspectrum}.

\begin{corollary}\label{coro.JfromLam}
    Let $\p,\p_1,\ldots,\p_n$ and $\J_\p(\D)$ be as in \Cref{thm.intro.jtspectrum}. Then we have that $\J_\p(\D) = \ov{ \bfL_\p(\G)}$.
\end{corollary}




\subsection{Other sets of interest}\label{sec.othersets}
Instead of looking at the value that $\D_\p$ takes over the spheres $S_R(T)$ we can ask what values $\D_\p$ takes over balls $B(T) = \{x\in\G: \p(o,x) < T\}$. Following the same proof as that in \Cref{thm.intro.jtspectrum} we can prove the following. 
\begin{proposition}
There exists a compact set $\J^B_{\p}(\D) \subset \R^n$ such that
\[
\frac{1}{T} \D(B(T)) \to \J^B_{\p}(\D)
\]
in the Hausdorff metric as $T\to\infty$. Furthermore, $\J^B_{\p}(\D)$ is convex.
\end{proposition}
Likewise we have the following.

\begin{proposition}
     We have that
    \[
\J^B_{\p}(\D) = \ov{\bigcup_{T\ge 0} \frac{1}{T}\bfL(B(T)) }.
    \]
\end{proposition}

\begin{proposition}\label{prop.convexhull}
    We have that $\J^B_{\p}(\D)$ is the convex hull $\textnormal{CH}(\J_{\p}(\D) \cup \{ \mathbf{0} \})$ where $\mathbf{0}$ represents the zero vector in $\R^n$.
\end{proposition}

The proof is simple but we include it for completeness.

\begin{proof}
We note that $\J_{\p}(\D) \cup \{\mathbf{0}\} \subset \J^B_{\p}(\D)$ and since $\J^B_{\p}(\D)$ is convex we deduce that $\textnormal{CH}(\J_{\p}(\D) \cup \{\mathbf{0}\}) \subset \J^B_{\p}(\D).$

Conversely, if $\vx \in \J^B_{\p}(\D)$ then we can find a sequence $x_T \in B(T)$ such that
\[
\frac{\D(x_T)}{T} \to \vx \ \text{ as $T\to\infty$.}
\]
By restricting to a subsequence we can also assume that
\[
\frac{\p(o,x_T)}{T} \to \alpha \in [0,1] \ \text{ as $T \to\infty$.}
\]
If $\alpha = 0$ then using that $\D_\p(x_T)$ is uniformly bounded in $\R^n$ we get
\[
\frac{1}{T}\D(x_T) \to \mathbf{0} \text{ as $T\to\infty$ and $\vx = \mathbf{0}$}.
\]
Otherwise $\alpha \in (0,1]$ and $\alpha^{-1} \vx \in \J_{\p}(\D)$. Hence $\vx$ belongs to the convex hull $\textnormal{CH}(\J_{\p}(\D) \cup \{\mathbf{0}\})$ as it lies on the line connecting $\mathbf{0}$ and $\alpha^{-1}\vx$ concluding the proof.
\end{proof}

\section{Manhattan Manifolds}\label{sec.manhattanmanifolds}
In this section we study the Manhattan manifold associated to a tuple of hyperbolic metric potentials. Here we prove the results stated in Sections \ref{sec.RWandD} and \ref{sec.mm}, except for \Cref{thm.homeointeriors} whose proof is postponed for \Cref{sec.metricstructuresmanhattan}. We fix
$\p_1,\dots,\p_n,\p$ as in the previous section. Recall that the parametrization $\thet_{\D/\p}$ of the Manhattan manifold for $\p_1,\dots,\p_n,\p$ associates to any $\vv\in \R^n$ the critical exponent of the series 
\[s \mapsto \sum_{x\in \G}{\exp(\-\langle \vv,\D(x) \rangle-s\p(o,x))}.\]

\subsection{Projecting the Manhattan manifold into $\calH_\G$} We use the parameterization $\theta_{\D/\p}$ to define hyperbolic metric potentials.



\begin{definition}\label{def.psi_v} 
    Let $\p_\vv\in \calH_\G$ be the potential defined as  
    \[\p_\vv(x,y)=\langle \vv, \D(x^{-1}y)\rangle+\thet_{\D/\p}(\vv)\p(x,y)=\sum_{i=1}^n{v_i\p_i(x,y)} +\thet_{\D/\p}(\vv)\p(x,y). \]
\end{definition}

A key property of the potentials $\p_\vv$ is the following.

\begin{lemma}\label{lem.EGR1}
    For each $\vv\in \R^n$, $\p_\vv$ belongs to $\calH^{++}_\G$ and has exponential growth rate 1.
\end{lemma}

Before proving this lemma we record some corollaries. In this section $O(x,R) = \{\x \in \partial \G: (x|\xi)_o^\p > \p(o,x) - R \}$ denotes the shadow set of radius $R>0$ based $x$ with respect to $\p$.

\begin{corollary}\label{coro.PSforv}
    Given $\vv = (v_1, \ldots, v_n) \in \R^{n}$, let $\nu_\vv$ be a quasi-conformal probability measure for $\p_\vv$. Then $\nu_\vv$ is Borel, quasi-invariant,  ergodic, and doubling respect to any visual quasi-metric on $\partial \G$. Moreover, it satisfies the following properties.
    \begin{enumerate}
        \item There exists $C >1$ depending only on $\vv$ and $R$ such that
        \[
        C^{-1} e^{-\p_\vv(o,x)} \le \nu_\vv(O(x,R)) \le         Ce^{-\p_\vv(o,x)} 
        \]
        for all $x \in \G$.
       \item We have that
       \[
       \theta_{\D/\p}(\vv) = \lim_{T\to\infty} \frac{1}{T} \log \left( \sum_{x \in S_R(T)} e^{-\langle \vv, \D(x)\rangle} \right).
       \]
    \end{enumerate}
    In addition, if $\hat\nu_\vv$ is a quasi-conformal measure for $\hat\p_{\vv}(x,y)=\p_{\vv}(y,x)$, then there exist $\hat C>1$ depending only on $\vv$ and $R$ such that  
        \[
        \hat{C}^{-1} e^{-\hat\p_\vv(o,x)} \le \hat\nu_\vv(O(x,R)) \le         \hat{C}e^{-\hat\p_\vv(o,x)} \text{ for all }x\in \G.
        \]
\end{corollary}

\begin{proof}
 By \Cref{lem.EGR1} we know that $\p_\vv \in \calH_\G^{++}$ and has exponential growth rate 1. Then the properties for $\nu_{\bfv}$ and $\hat\nu_\vv$ follow from \Cref{prop.existenceQCH_G}. Similarly, assertion (2) follows from the proof of \cite[Proposition 2.7]{cantrell-tanaka.manhattan}. 
\end{proof}

By \Cref{lem.EGR1}, for each $\vv \in \R^n$ let $\L_\vv\in \calC_\G$ denote a Bowen-Margulis current associated to $\p_\vv$. We also let $m_\vv$ be the corresponding flow invariant probability measure on the flow space $\calF_\k$, as in \Cref{sec.BMHMP}. 

\begin{corollary}\label{Cor:limit}
The assignment $\vv \mapsto m_\vv$ is continuous. That is, if $\vv \to \vv_0$ in $\R^n$ then $m_\vv$  converges to $m_{\vv_0}$ for the weak$^\ast$ topology.\end{corollary}

\begin{proof}
This follows from the same proof of Corollary 2.13 in \cite{cantrell-tanaka.manhattan}.
\end{proof}


We move on to the proof of \Cref{lem.EGR1}, which is an immediate consequence of the next proposition.

\begin{proposition}\label{prop.genEGR1}
    Let $\p_\ast \in \calH_\G$ and $\p\in \calH_\G^{++}$ and let $\thet\in \R$ be the critical exponent of the series 
    \[s \mapsto \sum_{x\in \G}{\exp(-\psi_\ast(o,x)-s\p(o,x))}\]
    for $s\in \R$. 
    Then $\p_\ast+\thet \p \in \calH_\G^{++}$ and has exponential growth rate 1.
\end{proposition}

A main ingredient in the proof of this result is the lemma below. Its proof follows the exact same argument as the proof of \cite[Lemma 3.7]{cantrell-reyes.manhattan}. The only difference is that in the current case the generating subsets $S$ appearing are not necessarily symmetric. Instead, for the final part of the argument we use the bound $v_S \geq \frac{1}{2}{\log(\al \# S)}$ valid for any generating subset $S\subset \G$, which follows from Theorem 1.1 and Remark 1.2 in \cite{delzant-steenbock}. We leave the details to the reader. In the next lemma, by a \emph{possibly asymmetric pseudo metric} on $\G$ we nonnegative function $d: \G \times \G \ra \R$ such that $d(x,x)=0$ and $d(x,z)\leq d(x,y)+d(y,z)$ for all $x,y,z\in \G$, but not necessarily satisfying $d(x,y)=d(y,x)$. The exponential growth rate of a possibly asymmetric pseudo metric is then defined as in \eqref{eq.defEGRtempered}.

\begin{lemma}\label{lem.convergencetoproper}
    Let $\G$ be a non-elementary hyperbolic group and let $(d_m)_m$ be a sequence of left-invariant possibly asymmetric pseudo metrics on $\G$ that pointwise converge to the possibly asymmetric pseudo metric $d_\infty$ on $\G$. If all the pseudo metrics $d_m$ have exponential growth rate 1, then $d_\infty$ is proper.
\end{lemma}

\begin{proof}[Proof of \Cref{prop.genEGR1}]
 For any $0\leq t\leq 1$ let $\thet(t)$ be the critical exponent of \[s \mapsto \sum_{x\in \G}{\exp(-t\psi_\ast(o,x)-s\p(o,x))}\]
and set $\psi_t:=t\psi_\ast \in \thet(t)\psi$. We will show that $\psi_t \in \calH_\G^{++}$ and has exponential growth rate 1 for all $t\in [0,1]$. The proof is divided into several steps.

\textbf{Claim 1.} $\psi_t \in \calH_\G^{+}$ for each $t$.

It is enough to prove it $t=1$. If $\psi_t=\psi_\ast +\thet \psi \notin \calH_\G^+$ then by \Cref{lem:MLSnonegative} and \Cref{coro.propertiesH} (2) we can find $C>0$ and $y\in \G$ such that $\ell_{\psi_1}(y)<0$ and \[\max\{|k\ell_{\p_1}(y)-\p_1(o,y^k)|,|k\ell_{\p}(y)-\p(o,y^k)|\}\leq C\] for all $k>0$. Then for any $s>0$ small enough so that $a:=-\ell_1(y)-s\ell_\p(y)>0$ we have
\begin{align*}
    \sum_{x\in \G}{\exp(-t\psi_\ast(o,x)-s\p(o,x))}  \geq \sum_{k>0}{\exp(-t\psi_\ast(o,y^k)-s\p(o,y^k))} 
     \geq e^{-C(1+s)} \sum_{k>0}{e^{ka}}=\infty,
\end{align*}
contradicting the definition of $\thet$. Therefore $\psi_1 \in \calH_\G^+$.

\textbf{Claim 2.} Let $a=\Dil(-\p_\ast,\p)$. Then $b_t:=\thet(t)-at\geq 0$ for al $t\in [0,1]$. 

We can assume that $t>0$. Since $\frac{1}{t}\psi_t\in \calH_\G^{+}$ by Claim 1, from \Cref{lem:MLSnonegative} we have that
\[0\leq \frac{1}{t}\ell_{\p_t}=\ell_{\p_\ast}+\frac{\thet(t)}{t}\ell_\p,\]
and hence $\thet(t)/t \geq a$ by the definition of $a$.

\textbf{Claim 3.} Fix $\hat \p,\hat\p^+\in \wh\calD_\G$ roughly isometric to $\p$ and $\p_\ast+a\p$ respectively (which exist by the definition of $a$ and \Cref{lem.GPpositiveH_G}). Then for any $t\in [0,1]$ the function 
$\hat\p_t:=(\thet(t)-at)\hat\p+t\hat\p^+$
belongs to $\wh\calD_\G$ and is roughly isometric to $\p_t$.

Claim 2 tells us that $\thet(t)-at \geq 0$, so that $\hat\p_t \in \wh\calD_\G$ for each $t$. Moreover, $\hat\p_t$ is roughly isometric to $(\thet(t)-at)\psi+t(\p_\ast+a\p)=\thet(t)\p+t\p_\ast=\p_t$.
\\





For the next claims we assume for the sake of a contradiction that the set $A:=\{t\in [0,1]: \psi_t \notin \calH_\G^{++}\}$ is non-empty. Let $t_0=\inf A$.  

\textbf{Claim 4.} $\psi_{t_0}\notin \calH_\G^{++}$, so in particular $t_0>0$. 

Suppose that $\psi_{t_0}\in \calH_\G^{++}$ and let $c>0$ be such that $\ell_{\psi}\geq c \ell_\p$. Also, let $\lam>0$ be such that $|\ell_{\p_\ast}|\leq \lam \ell_\p$, so by \Cref{lem:MLSnonegative} there exists $C>0$ such that $|\psi_\ast|\leq \p+C$. We will show that if $0<\ep \leq c/6\lam$ then $\p_{t_0+\ep}\in \calH_\G^{++}$, contradicting the definition of $t_0$. Suppose $\ep>0$ for which we have that 
\begin{align*}
    \sum_{x\in \G}{e^{-(t_0+\ep)\p_\ast(o,x)-(\thet(t_0)+2\ep\lam)\p(o,x)}}& \leq e^{\ep C}\sum_{x\in \G}{e^{-t_0\p_\ast(o,x)-(\thet(t_0)+\ep\lam)\p(o,x)}}.
\end{align*}
Then both series above converge by the definition of $\thet(t_0)$ and hence $\thet(t_0+\ep)\leq \thet(t_0)+2\ep\lam$, and a similar analysis gives us $\thet(t_0+\ep)\geq \thet(t_0)-2\ep\lam$. From this, for $\ep \leq c/6\lam$ we have that 
\begin{align*}
    \ell_{\p_{t_0+\ep}} = (t_0+\ep)\ell_{\p_\ast}+\thet(t_0+\ep)\ell_\p  
     \geq (t_0+\ep)\ell_{\p_\ast}+\thet(t_0)\p-2\ep\lam\ell_\p 
     \geq \ell_{\p_{t_0}}-3\ep\lam\ell_\p \geq \frac{c}{2}\ell_\p,
\end{align*}
and we deduce $\p_{t_0+\ep}\in \calH_\G^{++}$.

\textbf{Claim 5.} If $0\leq t<t_0$ then $\psi_t\in \calH_\G^{++}$ and has exponential growth rate 1. 

By the definition of $t_0$, if $t<t_0$ then $\psi_{t}\in \calH_\G^+$ and there exists $c>0$ such that $\p_t \geq c^{-1}\p-C$. Then for all $\ep>0$ we have
\[(1+\ep)\p_t \geq t\p_\ast+(\thet(t)+\ep c^{-1})\p-\ep c\]
and 
\[(1-\ep)\p_t \leq t\p_\ast+(\thet(t)-\ep c^{-1})\p+\ep c.\]
In particular, a similar reasoning as in the proof of Claim 4 implies that the series 
\[\sum_{x\in \G}{e^{-(1+\ep)\p_t(o,x)}} \text{ and } \sum_{x\in \G}{e^{-(1-\ep)\p_t(o,x)}}\]
are convergent and divergent respectively for all $\ep>0$, concluding that $\p_t$ has exponential growth rate 1.

\textbf{Claim 6.} $\psi_{t_0}\in \calH_\G^{++}$ (this is our desired contradiction as it conflicts with Claim 4).

It is enough to prove that $\hat\p_{t_0}\in \calH_\G^{++}$ for $\hat\p_{t_0}$ as in Claim 3. Note that $\hat\p_{t_0}$ is the pointwise limit of $\hat\p_{t}$ as $t \nearrow t_0$. By Claims 3 and 5, for each $t<t_0$ the function $\hat\p_{t}$ is a left-invariant (possibly asymmetric) pseudo metric on $\G$ with exponential growth rate 1. Then \Cref{lem.convergencetoproper} implies that $\hat\p_{t_0}$ is proper, so it belongs to $\calH_\G^{++}$ by \Cref{lem.charH_G^++}.\\

So far we have proven that $\p_t\in \calH_\G^{++}$ for all $t\in [0,1]$, and then the same argument as in the proof of Claim 5 gives us that $\p_t$ has exponential growth rate 1. 
\end{proof}

\subsection{Differentiability and convexity} In this section we study the regularity properties of the parametrization $\thet_{\D/\p}$, that we also denote by $\thet$ to simplify notation. Recall from the introduction that for $\nu$ an ergodic quasi-invariant Borel probability measure on $\partial \G$, its Lyapunov vector is denoted $\D_\p(\nu)$, whenever it exists.

We first prove differentiability of $\thet$.

%
\begin{proposition}\label{prop.C1}
    The parametrization $\theta$ is $C^1$.
\end{proposition}
\begin{proof}
Take a vector $\mathbf{v} = (v_1,\ldots, v_n) \in \R^n$ and consider the series
\[
(s,t) \mapsto \sum_{x\in\G} e^{-\langle (v_1+t, v_2, \ldots, v_n), \D(x)\rangle - s\p(o,x)}
\]
for $s,t \in \R$. For fixed $t$ we let $s = \theta_{\vv}(t)$ be the abscissa of convergence of this series. The function $\theta_{\vv} : \R \to \R$ is convex by H\"older's inequality and so is continuous and differentiable almost everywhere. Assuming that the derivative exists, we have that $\theta_\vv'(0)$ is the partial derivative in the first coordinate of $\theta$ at $\vv$. 
Recall that Cantrell and Tanaka \cite{cantrell-tanaka.manhattan} showed that Manhattan curves for pairs of metrics $d,d_\ast \in \Dc_\G$ are $C^1$. Although the curve $\theta_\vv$ is not a Manhattan curve, it is in essence a `weighted Manhattan curve' and the same techniques used by Cantrell and Tanaka in \cite{cantrell-tanaka.manhattan} can be applied. We therefore sketch the argument for the rest of this proof.

Consider the Mineyev flow constructed using a Green metric $\wh d$ as constructed in \Cref{sec.mf}.
Let $\vv_t = (v_1+t, v_2, \ldots, v_n)$ for $t \in \R$. First we note that by the same argument used in \cite[Lemmaa~3.5]{cantrell-tanaka.manhattan}, the (1-dimensional) Lyapunov vector ${(\p_1)}_{\p}(\nu_{\vv_t})$ exists.
This argument uses the fact that $\p_1$ and $\p$ are linear up to uniformly bounded constants along rough $\wh d$ geodesics by \Cref{prop.descriptionH_G}. As mentioned above, we know that $\theta_\vv$ is differentiable almost everywhere and it follows as in \cite[Lemma~3.6]{cantrell-tanaka.manhattan} that when $\theta_\vv'(t)$ exists, then it equals $-(\p_1)_{\p}(\nu_{\vv_t})$. 
However, using \Cref{Cor:limit} as in the proof of \cite[Theorem~3.7]{cantrell-tanaka.manhattan} we see that $\theta_\vv$ is in fact differentiable everywhere and in particular 
\begin{equation}\label{eq.components}
   \theta_\vv'(0)=\frac{\partial \thet}{\partial x_1}(\vv)= -(\p_1)_{\p}(\nu_{\vv_t}).
\end{equation}

Applying the same argument to the other components shows that the partial derivatives of $\theta$ exist everywhere. Moreover, using \Cref{Cor:limit} as in the proof of Lemma~3.5 in \cite{cantrell-tanaka.manhattan}, we can prove that these partial derivatives vary continuously. 
\end{proof}

We single out the following corollary which follows from combining the versions of \eqref{eq.components} for any partial derivative of $\thet$. 
\begin{corollary}\label{cor.drift}
          Let $\vv \in \mathbb{R}^n$ and and consider $\nu_\vv$ a quasi-conformal measure for $\p_{\vv}$. Then the Lyapunov vector $\D_\p(\nu_\vv)$ exists and equals  
    \[
    \D_\p(\nu_{\vv}) = -\nabla \theta(\vv).
    \]
    In particular, $\{-\nabla \theta(\vv): \vv\in \R^n\} \subset \calD\J_\p(\D)\subset \J_\p(\D)$.
\end{corollary}

We move on to prove the convexity of the Manhattan manifold.

\begin{proposition}\label{prop.convex}
    Suppose that $\p_1,\dots,\p_n,\p$ are independent. Then $\theta$ is strictly convex, and in particular the gradient function $\nabla \theta$ is injective. 
\end{proposition}

\begin{proof}
To prove this, we adapt the proof of \cite[Theorem~3.10]{cantrell-tanaka.manhattan} to our current situation.
 Suppose that $\theta$ is not strictly convex. Then, since $\theta$ is convex, there exist distinct vectors $\vv,\vw\in\R^n$ such that
 \[
 \theta\left(\frac{\vv+\vw}{2}\right) =  \frac{\left(\theta(\vv) + \theta(\vw)\right)}{2}.
 \]
 This follows from \Cref{prop.C1} since $\theta$ is $C^1$.
We then consider the vectors $\vv,\vw$ and $\vu := \frac{1}{2}(\vv+\vw)$ and note that the above equality implies that $\p_{\vu}=\frac{1}{2}(\p_\vv+\p_{\vw})$. In particular, the measures $\nu_\vv, \nu_\vw, \nu_\vu$ and $\hat\nu_\vv,\hat\nu_\vw,\hat\nu_\vu$ from \Cref{coro.PSforv} satisfy the shadow estimates needed to apply the proof of Theorem 3.10 in \cite{cantrell-tanaka.manhattan}. More precisely, from \Cref{coro.PSforv}(1)
we have that
\begin{equation} \label{eq.measure}
\frac{\nu_\vv(O(x,R))}{\nu_\vu(O(x,R))} \frac{\nu_\vw(O(x,R))}{\nu_\vu(O(x,R))} \ \text{ and } \  \frac{\hat\nu_\vv(O(x,R))}{\hat\nu_\vu(O(x,R))} \frac{\hat\nu_\vw(O(x,R))}{\hat\nu_\vu(O(x,R))}
\end{equation}
is uniformly bounded away from $0$ and $\infty$ for all $x\in\G$.

 Now, if $\L_\vv$ is a Bowen-Margulis current associated to $\p_\vv$ then \eqref{eq.defBMH_G} gives us
 \begin{equation}\label{eq.doubmeas}
\L_\vv = \varphi_\vv(\xi,\eta) e^{2 (x|y)^{\p_\vv}_o } \hat\nu_\vv \otimes \nu_\vv
 \end{equation}
where $\varphi_\vv : \partial \G^2 \to \R$ is uniformly bounded away from $0$ and infinity (see \cite[Example~2.9]{cantrell-tanaka.invariant}). There are similar expressions for the currents $\L_\vw$ and $\L_\vu$.

We now have the ingredients needed to apply the argument used in the proof of \cite[Theorem~3.10]{cantrell-tanaka.manhattan}. Using \cite[Lemma~3.11]{cantrell-tanaka.manhattan}, the fact that \eqref{eq.measure} is uniformly bounded, and the ergodicity of the Bowen-Margulis currents, we deduce that $\Lambda_\vv$ and $\Lambda_\vw$ are proportional, i.e. there is $c>0$ such that $\Lambda_\vv  = c\Lambda_\vw$ 
It follows from the definition of $\Lambda_\vv, \Lambda_\vw$ that $\hat\varphi = d\nu_\vv/d\nu_\vw$ satisfies
\[
\hat\varphi(\xi)\hat\varphi(\eta) e^{2(\x|\eta)_o^{\p_\vv} } \asymp e^{2(\x|\eta)_o^{\p_\vw} }
\]
for all $\xi,\eta \in \partial \G$. It then follows from \eqref{eq.doubmeas} and $c\Lam_\vv=\Lam_\vw$ that $\hat\varphi$ is bounded away from $0$ and infinity and that there exists a constant $C>0$ such that
\[
|\langle \vv, \D(x)\rangle  + \theta(\vv)\p(o,x) - \langle \vw, \D(x)\rangle - \theta(\vw)\p(o,x)| \le C
\]
for all $x \in \G$, i.e. $\p_1, \ldots, \p_n, \p$ are dependent. 
This contradicts our assumptions and the result follows.
    \end{proof}

We have all the ingredients to prove \Cref{thm.manc^1}.

\begin{proof}[Proof of \Cref{thm.manc^1}]
 From \Cref{prop.C1} and \Cref{prop.convex} we have that $\thet$ is $C^1$ and strictly convex when $\p_1,\dots,\p_n,\p$ are independent. Under this assumption, we readily have that $\nabla \thet$ is injective. Consider then the function
    \[
    (\vv,y) \mapsto \Theta(\vv,y) := \theta(\vv) - y
    \]
    on $\R^n \times \R$. This map is $C^1$ by \Cref{prop.C1} and so the implicit function theorem for manifolds implies that $\calM_{\D/\p}=\Theta^{-1}(0)$ is an $n$-dimensional $C^1$-manifold. 
\end{proof}

Convexity and differentiability of $\nabla\theta$ allows us to relate the Manhattan manifold and the joint translation spectrum. For that, we first note a deviation principle for $\D_\p$. Suppose that $\p_1,\dots,\p_n,\p$ are independent and let $I : \R^n \to \R\cup \{\infty\}$ be the function defined by
\[
I(\vv) = v_\psi + \sup_{\vw \in \R^n} (\langle \vw, \vv \rangle - \theta(-\vw) ).
\]
Since $\nabla\theta$ is injective, for $\vv \in \J_\p(\D)$ we have that
    \[
    I(\vv) = v_\p + \langle \vv,(-\nabla \theta)^{-1}(\vv) \rangle - \theta((-\nabla \theta)^{-1}(\vv)).
    \]
By \Cref{coro.PSforv}(2) and the fact that $\theta$ is $C^1$ we can apply the the G\"artner-Ellis Theorem and deduce the following.
    
\begin{corollary}\label{cor.ldp}
Suppose that $\p, \p_1,\ldots,\p_n$ are independent.  Then, for any open $U \subset \R^n$ and closed $V \subset \R^n$ with $U \subset V$ we have that
\begin{align*}
    -\inf_{\vv \in U} I(\vv) &\le \liminf_{T\to\infty} \frac{1}{T} \log\left(  \frac{\#\left\{ x \in \G: \p(o,x) < T, \ \ \D_\p(x) \in U \right\}}{\#\{x \in \G : \p(o,x) <T \}} \right) \\
    &\le \limsup_{T\to\infty} \frac{1}{T} \log\left(  \frac{\#\left\{ x \in \G: \p(o,x) < T, \ \ \D_\p(x) \in V \right\}}{\#\{x \in \G : \p(o,x) <T \}} \right)\\
    &\le -\inf_{\vv\in V} I(\vv).
\end{align*}
Furthermore, $I$ is finite on $\nabla \theta (\R^n)$ and infinite on the complement of  $\ov{\nabla \theta (\R^n)}$.
\end{corollary}

This large deviation result allows us to recover the joint translation spectrum from the gradient of $\thet$ and prove one of the inclusions in \Cref{thm.homeointeriors}.
\begin{proposition} \label{prop.deriv=dj}
    Suppose that $\p_1,\dots,\p_n,\p$ are independent. Then we have that
    \[
 \J_{\p}(\D) = \ov{\left\{-\nabla \theta(\vv): \vv \in \R^n \right\}}
    \]
    and  that $\{-\nabla \theta(\vv): \vv \in \R^n \}$ is an open subset of $\R^n$ contained in $\Int(\J_{\p}(\D)).$
\end{proposition}

\begin{proof}
    One implication of the first equality is simple as $\ov{\left\{-\nabla \theta(\vv): \vv \in \R^n \right\}} \subset \J_{\p}(\D)$
follows from \Cref{cor.drift}.

The other inclusion of the equality follows from the large deviation principle.
Let $I$ be the rate function from \Cref{cor.ldp} and suppose that $\vx \in \J_{\p}(\D)$. Then there exists a sequence $(x_k)_k$ in $\G$ with
\[
\D_\p(x_k) \to \vx \text{ as $k\to\infty.$}
\]
Then by the large deviation principle we deduce that for any closed set $U \subset \R^n$ with $\vx \in U$,  
\[
-v_{\p} \le \limsup_{T\to\infty} \frac{1}{T} \log \left( \frac{1}{\#S_R(T)} \#\left\{ x \in S_R(T) : \D_\p(x) \in U \right\} \right) \le -\inf_{\vu \in U} I(\vu)
\]
and so $\inf_{\vu \in U} I(\vu) < \infty$. This implies that $\vx \in \ov{-\{\nabla \theta(\R^n)}$ by \Cref{cor.ldp}.

The expected inclusion follows from the fact that $\{-\nabla \thet(\vv): \vv \in \R^n\}$ is open in $\R^n$, which can be deduced from \Cref{prop.convex} and Invariance of domain \cite{brouwer}.
\end{proof}




From the the proposition above we can already deduce \Cref{prop.interiorlambda}.

\begin{proof}[Proof of \Cref{prop.interiorlambda}]
    From \Cref{prop.deriv=dj} we already have $-\nabla\thet(\mathbf{0}) \in \Int(\J_\p(\D))$, so we are only left to show that this vector equals the mean distortion vector $\tau_{\D/\p}$. By \Cref{cor.drift}, it is enough to prove that $\tau_{\D/\p}$ equals the Lyapunov vector $\D_\p(\nu_\p)$, where $\nu_\p=\nu_{\mathbf{0}}$ is a quasi-conformal measure for $\p$. This can be proved exactly as Theorem 3.12 in \cite{cantrell-tanaka.manhattan}. We leave the details to the reader.
\end{proof}


\subsection{Continuity properties}

In this section we analize how the Manhattan manifold and the joint translation spectrum vary as we move the input (rough similarity classes of) hyperbolic metric potentials in $\scrH^{++}_\G$. For that, we normalize as follows.

\begin{definition}
    Given $\p_1, \ldots, \p_n \in \calH_\G$ and $\rho=[\p]\in \scrH_\G^{++}$ we define
    \[
    \J_{\rho}(\p_1,\ldots,\p_n)
    \]
     as the joint translation spectrum of $\p_1, \ldots, \p_n$ with respect to $\p$, where $\p\in \rho$ is chosen so that $v_{\p}=1$. Similarly, given $\rho_1,\dots,\rho_n,\rho \in \scrH_\G^{++}$ we define
     \[
    \J_{\rho}(\rho_1,\ldots,\rho_n)
    \]
as the joint translation spectrum of $\p_1,\dots,\p_n$ with respect to $\p$, where $\p\in \rho,\p_1\in \rho_1,\dots,\p_n \in \rho_n$ are chosen so that $v_\p=v_{\p_1}=\cdots =v_{\p_n}=1$.
\end{definition}
Similarly, for tuples $\D=(\p_1,\dots,\p_n)$ and $[\D]=(\rho_1,\dots,\rho_n)$ and $\rho\in \scrH_\G^{++}$ as above, we consider the parametrizations $\thet_{\D/\rho}$ and $\thet_{[\D]/\rho}$ of the corresponding Manhattan manifolds.

In the next proposition, compact subsets of $\R^n$ are equipped with the Hausdorff topology and and continuous functions from $\R^n$ into $\R$ are equipped with the uniform topology on compact subsets. Recall that $\scrH_\G^{++}$ is always equipped with the topology induced by the Thurston metric $\Del$ as in \Cref{sec.HMPmetricstructures}.
\begin{proposition}\label{prop.continuity}
    \begin{enumerate}
        \item For fixed $\p_1,\dots,\p_n \in \calH_\G$, the assignment $$\rho \mapsto \J_\rho(\p_1,\dots,\p_n)$$ defines a continuous function from $\scrH_\G^{++}$ into the space of compact subsets of $\R^n$. Similarly, the assignment $\rho \mapsto\thet_{\D/\rho}$ is a continuous map $\scrH_\G^{++}$ into the space of continuous functions from $\R^n$ into $\R$.
        \item The assignment $$(\rho_1,\dots,\rho_n,\rho) \mapsto \J_\rho(\rho_1,\dots,\rho_n)$$ defines a continuous function from $(\scrH_\G^{++})^{n+1}$ into the space of compact subsets of $\R^n$. Similarly, the assignment $(\rho_1,\dots,\rho_n,\rho) \mapsto \thet_{[\D]/\rho}$ is a continuous map from $\scrH_\G^{++}$ into the space of continuous functions from $\R^n$ into $\R$.
    \end{enumerate}
 
\end{proposition}
\begin{proof}
The continuity of the joint translation spectra follows immediately from the definition of $\Del$ and \Cref{thm.intro.jtspectrum}. For the continuity of the Manhattan curves, the same proof of \cite[Lemma 3.8]{cantrell-reyes.MLS} applies in this case.
\end{proof}


\subsection{Dynamical and random walk translation spectra}

In this section we study the dynamical translation spectrum and random walk spectrum defined in \Cref{sec.RWandD}. We first prove \Cref{thm.dj=j}.

\begin{proof}[Proof of \Cref{thm.dj=j}]
One implication is easy since we clearly have $\mathcal{DJ}_{\p}(\D) \subset \J_{\p}(\D)$. For the other direction we note that by \Cref{prop.existenceQCH_G} and \Cref{cor.drift} we have that
\[
-\nabla \theta(\R^n) \subset \mathcal{DJ}_{\p}(\D).
\]
Taking closures and applying \Cref{prop.deriv=dj} gives the result. The moreover statement will follow from \Cref{prop.IntJ}.
\end{proof}

We continue with the proof of \Cref{thm.wj=j}, so we assume $\p_1=d_1,\dots,\p_n=d_n$ and $\p=d$ all belong to $\calD_\G$. 


\begin{proof}[Proof of \Cref{thm.wj=j}]
Clearly 
\[
\mathcal{WJ}_{d}(d_1,\ldots,d_n) \subset \J_{d}(d_1,\ldots,d_n)
\]
and taking closures gives one of the desired implications. We therefore just need to show the converse. We will in fact show that
    \[
\{ -\nabla \theta(\vv) : \vv \in \R^n \} \subset \overline{\mathcal{WJ}_{d}(d_1,\ldots,d_n)}
    \]
as taking the closure of this expression gives the needed result by \Cref{prop.deriv=dj}. 
Given any $\vv \in \R^n$, since $\p_\vv$ is symmetric, in fact there exists $d_\vv\in \calD_\G$ roughly isometric to $\p_\vv$. Let $\nu_\vv$ be a quasi-conformal measure for $d_\vv$.

Let us fix $\vv\in \R^n$ and note the following. By \Cref{cor.drift}, for any $\p_\ast\in \calH_\G$ the (1-dimensional) Lyapunov vector $(\p_\ast)_{\p_{\vv}}(\nu_\vv)$ equals $-\thet'_{\p_\ast/\p_\vv}(0)$, where 
where $\thet_{\p_\ast/d_\vv}$ is the Manhattan curve for $\p_\ast,d_\vv$. Applying this to the metrics $d,d_1,\dots,d_n$, we have the identities 
    \small{\begin{equation}\label{eq.manhattangreen1}
   \D_{d_\vv}(\nu_\vv)=-(\thet'_{d_1/d_\vv}(0),\dots,\thet'_{d_n/d_\vv}(0)), \  (d)_{d_\vv}(\nu_\vv)=-\thet'_{d/d_\vv}(0) \ \text{ and } \ \D_d(\nu_{\vv})=((d)_{d_\vv})^{-1}\D_{d_\vv}(\nu_\vv),
    \end{equation}}
\normalsize where the last identity follows from the definition of Lyapunov vectors.

Now, by \Cref{thm.gmd} there exists a sequence $\lam_m$ of admissible probability measures on $\G$ such that the Green metrics $d_{\lam_m}$ satisfy $[d_{\lam_m}] \to [d_\vv]$ in $(\scrD_\G,\Del)$. Since Green metrics have exponential growth rate 1 \cite{BHM-greenspeed}, and similarly for $d_\vv$ by \Cref{lem.EGR1}, \Cref{thm.intro.cone} implies that the Manhattan curves $\theta_{d_1/d_{\lam_m}}, \ldots, \theta_{d_n/d_{\lam_m}}, \theta_{d/d_{\lam_m}}$ converge uniformly on compact sets as $m \to\infty$ to $\theta_{d_1/d_\vv}, \ldots, \theta_{d_n/d_\vv}, \theta_{d/d_\vv}$ respectively. 
It follows that
    \begin{equation}\label{eq.manhattangreen2}
        \theta_{d_j/d_{\lam_m}}'(0) \to \theta_{d_j/d_\vv}'(0) \ \text{ and } \ \theta_{d/d_{\lam_m}}'(0) \to \theta_{d/d_\vv}'(0)
    \end{equation}
    for each $j=1,\ldots, n$ as $k\to \infty$.

Moreover, for each $m$, a quasi-conformal measure for $d_{\lam_m}$ is in the same measure class as the hitting measure of the random walk associated to $\lam_m$. This implies that for any $d_\ast \in \calD_\G$ and for a typical $\lam_m$-random walk $(Z_{k,\lam_m})_k$ we have
\[\lim_{k\to \infty}\frac{d_\ast(o,Z_{k,\lam_m})}{d_{\lam_m}(o,Z_{k,\lam_m})}=-h_{\lam_m}\thet'_{d_\ast/d_{\lam_m}}(0),\]
where $h_{\lam_m}$ is the entropy of $\lam_m$. 

For each $m$, this implies that the vector $$(\thet'_{d/d_{\lam_m}}(0))^{-1}(\thet'_{d_1/d_{\lam_m}}(0),\dots,\thet'_{d_n/d_{\lam_m}}(0))$$
belongs to the random walk spectrum $\calW\J_d(\D)$.
But from \eqref{eq.manhattangreen1} and \eqref{eq.manhattangreen2}, these vectors converge to $\D_d(\nu_\vv)$, which equals $-\nabla\thet(\vv)$ by \Cref{cor.drift}.
Therefore we deduce $-\nabla \thet_{\D/d}(\vv) \in \ov{\calW\J_d(\D)}$, as desired.
\end{proof}

\section{The point of view of metric structures}\label{sec.metricstructuresmanhattan}

Let $\p,\p_1,\dots,\p_n$ as in the previous sections. In this section we extend the work in \cite{cantrell-reyes.manhattan} and use the Manhattan manifold with parameterization $\thet=\thet_{\D/\p}$ to construct a subset of $\scrH_\G^{++}$ homeomorphic to $\R^n$. We will also bordify this set by considering appropriate boundary metric structures. Here we finish the proof of \Cref{thm.homeointeriors} and prove the result from \Cref{subsec.intro.manhattanmetric}.






\subsection{Projecting the Manhattan manifold into $\scrH_\G^{++}$}\label{sec.Manhattan+boundary}

As discussed in the introduction, we can also use $\theta$ to see the Manhattan manifold as a subset of $\scrH_\G^{++}$.

\begin{definition}\label{def.manhattanmanifoldmetric}
 We define $\scrM_{\D/\p}\subset \scrH_\G^{++}$ as the set of all metric structures of the form $\rho^{\D/\p}(\vv):=[\p_\vv]$ for $\vv\in \R^n$, where $\p_\vv\in \calH_\G$ is as in \Cref{def.psi_v}. We also call this set the \emph{Manhattan manifold} for $\p_1,\dots,\p_n$ with respect to $\p$.
\end{definition}

From the continuity of $\thet$ it follows that the map $\vv \mapsto \rho^{\D/\p}(\vv)$ is continuous. Moreover, if $\p,\p_1,\dots,\p_n$ are independent then $\rho^{\D/\p}$ injective. In fact, the following is true, confirming the first assertion in \Cref{prop.fromManhattantoJTS}.

\begin{proposition}\label{prop.manhattanhomeo}
 Suppose that $\p,\p_1,\dots,\p_n$ are independent. Then the map $\rho^{\D/\p}:\R^n \ra \scrM_{\D/\p}$ is a homeomorphism.
\end{proposition}

To prove this result we need to introduce the boundary metric structures that describe the limits of divergent sequences metric structures in $\scrM_{\D/\p}$. Recall that  $\|\cdot\|$ is the Euclidean norm on $\R^n$ and let $\bbS^{n-1}$ be the unit sphere, also seen as the set of rays in $\R^n$ based at the origin. For a non-zero vector $\vv\in \R^n$ we denote $\hat \vv=\frac{1}{\|\vv\|} \vv\ \in \bbS^{n-1}$. Given $\vv\in \R^n$ we use the notation $\D_\vv(x,y)$ for the potential  
\[ 
(x,y) \mapsto \langle \vv, \D(x^{-1}y) \rangle.
\]
We note that $\D_\vv \in \calH_\G$ and that $\psi_\vv=\D_\vv+\thet(\vv)\psi$ for all $\vv$. 

\begin{definition}\label{def.boundaryHMP}
    Let $\vv\in \R^n$ be non-zero. We define $\psi_{[\vv]}\in \calH_\G$ according to  
    \[\psi_{[\vv]}:=\D_\vv+\Dil(\D_{-\vv},\psi)\psi.\]
\end{definition}

Let $\partial \calH_\G^{++}$ be defined as the set of all the unbounded hyperbolic metric potentials in $\calH_\G^{+}\bs \calH_\G^{++}$ and let $\partial \scrH_\G^{++}$ be the corresponding set of rough similarity classes. We think of $\partial \scrH_\G^{++}$ as the boundary of $\scrH_\G^{++}$, in analogy to the Manhattan boundary discussed in \cite{cantrell-reyes.manhattan}.  

It is not hard to see that if $\vv$ and $\vw$ represent the same direction in $\mathbb{S}^{n-1}$ then $\p_{[\vv]}$ and $\p_{[\vw]}$ represent the same point in $\partial \scrH_\G^{++}$. Conversely, we have the following.
 
\begin{lemma}\label{lem.extensionboundary}
    Suppose that $\p_1,\dots,\p_n,\p$ are independent. If $\vv \in \R^n$ is non-zero then $\p_{[\vv]}\in \partial \calH_\G^{++}$. Moreover, if $\vv,\vw\in \R^n$ are non-zero then $[\psi_{[\vv]}]=[\psi_{[\vw]}]$ in $\partial \scrH_\G^{++}$ if and only if $\hat\vv=\hat\vw$. 
\end{lemma}

\begin{proof}
    By the definition of $\p_{[\vv]}$ and $\Dil(\D_{-\vv},\p)$ and the fact that $\p\in \calH_\G^{++}$, we have that $\p_{[\vv]}\in \calH_\G^{+}\bs \calH_\G^{++}$, and independence of $\p_1,\dots,\p_n,\p$ implies that $\p_{[\vv]}$ is unbounded, so it belongs to $\partial \calH_\G^{++}$. Independence also implies $\hat\vv=\hat\vw$ whenever $[\psi_{[\vv]}]=[\psi_{[\vw]}]$.    
\end{proof}

The definition of $\p_{[\vv]}$ is motivated by the next lemma. 

\begin{lemma}\label{lem.limitthetaboundary}
    Suppose that $\p_1,\dots,\p_n,\p$ are independent. If $\vv_m$ is a diverging sequence in $\R^n$ such that $\hat\vv_m$ converges to $\hat\vv\in \mathbb{S}^{n-1}$, then we have that 
    \[\lim_{m\to \infty} \frac{\thet(\vv_m)}{\|\vv_m\|}=\Dil(\D_{-\hat\vv},\psi).\]
\end{lemma}

\begin{proof}
    \Cref{prop.deriv=dj} and \Cref{thm.manc^1} imply that $\thet$ is Lipschitz, and hence there exists $C>0$ such that $|\thet(\vv)|\leq C\|\vv\|+|\thet(\mathbf{0})|$ for all $\vv$. 

    We can assume that $\vv_m$ is non-zero for every $m$. By \Cref{lem.EGR1} we have $\psi_{\vv_m}\in \calH_\G^{++}$, and since $\psi_{[\vv_m]} \in \partial \calH_\G^{++}$ by \Cref{lem.extensionboundary}, we must have
    $$\D_{\vv_m}+\thet(\vv_m)\psi=\psi_{\vv_m}\geq \psi_{[\vv_m]}=\D_{\vv_m}+\Dil(\D_{-\vv_m},\psi)\psi.$$
    This implies that $\thet(\vv_m)\geq \Dil(\D_{-\vv_m},\psi)$ for all $m$ and hence dividing by $\|\vv_m\|$ and letting $m$ diverge we obtain 
   \[\liminf_{m\to \infty} \frac{\thet(\vv_m)}{\|\vv_m\|} \geq \liminf_{m \to \infty}{\frac{\Dil(\D_{-\vv_m},\psi)}{\|\vv_m\|}}=\Dil(\D_{-\hat\vv},\psi).\]

    To prove the other needed inequality, suppose that up to extracting a subsequence we have that $\frac{\thet(\vv_m)}{\|\vv_m\|} \geq \Dil(\D_{-\hat\vv},\psi)+c$ for all $m$, for some $c>0$ independent of $m$. Then for all $m$ we have
    \begin{align*}
        \psi_{\vv_m}& =\D_{\vv_m}+\thet(\vv_m)\psi \\
        & \geq \D_{\vv_m}+\|\vv_m\|(\Dil(\D_{-\hat\vv},\psi)+c)\psi \\
        & = \|\vv_m\|\psi_{[\hat\vv]}+\|\vv_m\|(c\psi+\D_{\vw_m}) \\
        & \geq \|\vv_m\|(c\psi+\D_{\vw_m})
    \end{align*}
    for $\vw_m=\hat\vv_m-\hat\vv$. 
    Since $\vw_m$ tends to zero and $\psi$ is proper, we can find $c'>0$ such that $c\psi+\D_{\vw_m}\geq c'\psi$ for all $m$ large enough, so that $\psi_{\vv_m}\geq \|\vv_m\|c'\psi$ for all $m$ large. But $\|\vv_m\|$ diverges, so this would imply that the exponential growth rate of $\psi_{\vv_m}$ tends to zero. As this contradicts \Cref{lem.EGR1} we deduce that
    \[\limsup_{m\to \infty} \frac{\thet(\vv_m)}{\|\vv_m\|} \leq \Dil(\D_{-\hat\vv},\psi),\]
    as desired.
\end{proof}

\begin{proof}[Proof of \Cref{prop.manhattanhomeo}]
    As discussed above, the independence of $\p_1,\dots,\p_n,\p$ implies that $\rho^{\D/\p}$ is continuous and injective. To show that the inverse map $(\rho^{\D/\p})^{-1}: \scrM_{\D/\p} \ra \R^n$ is continuous, let $\vv_m$ be a sequence such that $\rho^{\D/\p}(\vv_m)$ converges to $\rho^{\D/\p}(\vv_\infty)$ for some $\vv_\infty\in \R^n$. We claim that $\vv_m \to \vv_\infty$. First we show that $\vv_m$ is a bounded sequence. Otherwise, after extracting a subsequence we can assume that $\vv_m$ diverges and $\hat \vv_m$ converges to $\hat \vv$ in $\mathbb{S}^{n-1}$. Then \Cref{lem.limitthetaboundary} implies that the sequence of translation length functions $\frac{1}{\|\vv_m\|}\ell_{\p_{\vv_m}}$ pointwise converges to the translation length function $\ell_{\p_{[\hat \vv]}}$. This is impossible by the convergence $\rho^{\D/\p}(\vv_m) \ra \rho^{\D/\p}(\vv_\infty)$ and the fact that each $\p_{\vv}$ has exponential growth rate 1. Therefore, the sequence $\vv_m$ is bounded and it is enough to show that $\vv_\infty$ is its only accumulation point. But if $\vv_m$ converges to $\vv$, then continuity of $\rho^{\D/\p}$ and independence of $\p_1,\dots,\p_n,\p$ gives us $\vv=\vv_\infty$. This completes the proof of the proposition.
\end{proof}


\subsection{Boundary of the joint translation spectrum}
We move on to study the boundary of joint translation spectrum, which can be used via the boundary potentials $\p_{[\vv]}$ from \Cref{def.boundaryHMP}. Before proceeding to this, we first note a useful characterization of the joint translation spectrum in terms of geodesic currents. From \Cref{thm.continuousextension}, we see that the function $\bfL_\p$ in \eqref{eq.defDlam} extends (uniquely) to a continuous function $\bfL_\p:\bbP\calC_\G \ra \R^n$, where $\bbP\calC_\G$ is the space of projective geodesic currents.  From this extension and \Cref{coro.JfromLam} we easily deduce the following.

\begin{lemma}\label{lem.currentsJTS}
For any $\p_1,\ldots,\p_n \in \calH_\G$ and $\p\in \calH_\G^{++}$ we have
\[\J_\p(\D)=\bfL_\p(\bbP\calC_\G).\]
\end{lemma}


Throughout the rest of the section we assume that $\p_1,\dots,\p_m,\p$ are independent. Given $\vx \in \J_\p(\D)$, we let $N_\J(\vx)$ be the set of all the projective classes $[\vv]\in \mathbb{S}^{n-1}$ such that $\langle \vx-\vy, \vv\rangle \leq 0$ for all $\vy \in \J_\p(\D)$. That is, $[\vv]$ belongs to $N_\J(\vx)$ if for the hyperplane $H$ containing $\vx$ and normal to $\vv$, $\J_\p(\D)$ is contained in the half-space determined by $H$ in the direction of $-\vv$. In particular, $N_\J(\vx)$ is non-empty if and only if $\vx \in \partial \J_\p(\D)$. From \Cref{lem.currentsJTS} we have the following characterization for vectors in $N_\J(\vx)$.

\begin{lemma}\label{lem.criterionnormal}
    For any $\vx\in \J_\p(\D)$ we have
    \[N_\J(\vx)=\{[\vv]\in \mathbb{S}^{n-1}: \ell_{[\vv]}(\mu)=0 \text{ for some } \mu\in \calC_\G \text{ with }\bfL_\p(\mu)=\vx\}.\]
    Moreover, if $[\vv]\in N_\J(\vx)$ then $\ell_{[\vv]}(\mu)=0$ for any $\mu\in \calC_\G$ with $\bfL_\p(\mu)=\vx$. 
\end{lemma}
This lemma yields the following criterion for a point in the joint translation spectrum to belong to its boundary.

\begin{corollary}\label{coro.criterionboundary}
   If $\vv \in \J_\p(\D)$, we have that $\vx\in \partial \J_\p(\D)$ if and only there exist non-zero $\vv\in \R^n$ and non-zero $\mu\in \calC_\G$ such that $\vx=\bfL_\p(\mu)$ and $\ell_{\psi_{[\vv]}}(\mu)=0$.
\end{corollary}

\begin{proof}[Proof of \Cref{lem.criterionnormal}]
    Suppose that $\vx=\bfL_\p(\mu)$ for $\mu\in \calC_\G$ non-zero, which holds by \Cref{lem.currentsJTS}. For a non-zero vector $\vv \in \R^n$ we claim that $\ell_{\psi_{[\vv]}}(\mu)=0$ if and only if $\langle \vx-\vy, \vv\rangle \leq 0$ for all $\vy \in \J_\p(\D)$. This implies the lemma since the condition $\langle \vx-\vy, \vv\rangle \leq 0$ for all $\vy\in \J_\p(\D)$ is precisely $[\vv] \in N_\J(\vx)$.
    
    Now, denote $D=\Dil(\D_{-\vv},\psi)$. Since $\p_{[\vv]} \in \calH_\G^+ \bs \calH_\G^{++}$ by \Cref{lem.extensionboundary}, then \Cref{coro.charH+currents} tells us that  $\ell_{\psi_{[\vv]}}(\mu)=0$ is equivalent to
    \[
    \frac{\ell_{\psi_{[\vv]}}(\mu)}{\ell_\psi(\mu)}\leq \frac{\ell_{\psi_{[\vv]}}(\mu')}{\ell_\psi(\mu')} \text{ for all }\mu' \in \bbP\calC_\G,
    \]
    or 
    \[
    \frac{\langle \bfL(\mu),\vv \rangle+D\ell_{\p}(\mu)}{\ell_{\psi}(\mu)}\leq \frac{\langle \bfL(\mu'),\vv \rangle+D\ell_{\p}(\mu')}{\ell_\psi(\mu')} \text{ for all }\mu' \in \bbP\calC_\G.
    \]
    This is the same as $ \langle \bfL_\p(\mu),\vv \rangle \leq \langle \bfL_\p(\mu'),\vv \rangle$ for all $\mu' \in \bbP\calC_\G$, and hence $\langle \vx,\vv \rangle \leq \langle \vy,\vv \rangle $ for all $\vy \in \J_\p(\D)$ by \Cref{lem.currentsJTS}.
\end{proof} 

\begin{remark}\label{remark.fillingcurrents}
A geodesic current $\mu \in \calC_\G$ with full support verifies $\ell_{\p'}(\mu)>0$ for all $\p'\in \partial\calH_\G^{++}$. Therefore, the lemma above implies that $\bfL_\p(\mu)\in \Int(\J_\p(\D))$ for any geodesic current $\mu$ with full support.
\end{remark}

We also need the next lemma. Recall that we are assuming $\p_1,\dots,\p_n,\p$.

\begin{lemma}\label{lem.ineqconvexity}
    For any $\vv\in \R^n$ we have
    \begin{equation}\label{eq.boundstheta}\theta(\bfv)-v_\psi \leq \langle \bfv, \nabla \theta (\bfv)\rangle < \theta(\bfv),
    \end{equation}
where the first inequality is strict if $\vv\neq \mathbf{0}$.
\end{lemma}

\begin{proof}
    By \Cref{prop.convex} for all $\bfv\neq \mathbf{0}$ in $\R^n$ we have \[\
\thet(\vv)-v_\psi=\thet(\vv)-\thet(\mathbf{0}) <\langle \bfv -\mathbf{0}, \nabla \theta (\bfv)\rangle = \langle \bfv, \nabla \theta (\bfv)\rangle,
 \]
    which proves the first inequality in \eqref{eq.boundstheta}.  For the second inequality we fix $\vv\in \R^n$ and let $\ov\thet=\ov\thet_{\p/\p_\vv}$ be the parameterization of the Manhattan curve of $\p$ over $\p_\vv$. By \Cref{lem.EGR1}, for any $s$ the number $\ov\thet(s)$ is characterized by the property that the function 
    $\ov\p_s:=s\p+\ov\theta(s)\p_\vv$ belongs to $\calH^{++}_\G$ and has exponential growth rate 1. But
    \[\ov\p_s:=s\p+\ov\thet(s)\p_\vv=s\p+\ov\thet(s)(\D_\vv+\thet(\vv)\p)=\D_{\ov\thet(s)\vv}+(s+\ov\thet(s)\thet(v))\p,\]
    and hence the definition of $\thet$ and \Cref{lem.EGR1} imply that
    \[\thet(\ov\theta(s)\vv)=s+\ov\thet(s)\thet(\vv)\]
    for all $s\in \R$.
    By \Cref{thm.manc^1}, both $\thet$ and $\thet'$ are differentiable, so taking derivatives with respect to $s$ yields
     \[\langle \nabla\theta(\ov\theta(s)\vv),\ov\theta'(s)\vv \rangle =1+\ov\thet'(s)\theta(\vv).\]
     At $s=0$, and using $\ov\thet(0)=v_{\p_\vv}=1$ by \Cref{lem.EGR1}, we obtain
     \[1=(\langle \nabla\thet(\vv),\vv \rangle-\thet(\vv))\ov\thet'(0),\]
     and since $\ov\thet'(0)<0$ (both $\p$ and $\p_\vv$ belong to $\calH_\G^{++}$), we get $\thet(\vv)>\langle \nabla\thet(\vv),\vv\rangle$, as desired.
    \end{proof}

Using this lemma we can finish the proof of \Cref{thm.homeointeriors}, which combined with \Cref{prop.manhattanhomeo} completes the proof of \Cref{prop.fromManhattantoJTS}.

\begin{proposition}\label{prop.IntJ}
    The map $-\nabla\thet:\R^n \ra \Int(\J_\p(\D))$ is a homeomorphism.
\end{proposition}

\begin{proof}
\Cref{thm.manc^1}, \Cref{prop.deriv=dj} and Invariance of domain gives us that $-\nabla\theta:\R^n \ra \Int(\J_\p(\D))$ is well-defined and is a homeomorphism onto its image, which is an open subset of $\Int(\J_\p(\D))$. Hence we are left to show that $\Int(\J_\p(\D))\subset -\nabla\theta(\R^n)$.

    To show this, let $A=-\nabla\thet(\R^n)$ and suppose for the sake of contradiction that there exists $\vx \in \Int(\J_\p(\D)) \bs A$ and let $\gam=[-\nabla\thet(\mathbf{0}),\vx]$ be the segment joining $\vx$ and $-\nabla\thet(\mathbf{0})$. Let $\vy \in \gam$ be the nearest point to $\vx$ such that $\vy' \in A$ for any $\vy'\in \gam$ closer to $-\nabla\thet(\mathbf{0})$ than $\vy$. Since $\Int(\J_\p(\D))$ is convex we have $\vy\in \Int(\J_\p(\D))$, and since $A$ is open we have $\vy\notin A$. In particular, there exists a sequence $\vv_m \in \R^n$ such that $-\nabla\thet(\vv_m) \to \vy$ as $m$ diverges. 
    
    The condition $\vy \notin A$ implies that $\vv_m$ diverges as $m$ tends to infinity, so up to taking a subsequence suppose that $\hat\vv_m \to \hat\vv \in \bbS^{n-1}$. Given $\mu \in \calC_\G$ such that $\vy \in \bfL_\p(\mu)$, we claim that $\ell_{[\hat\vv]}(\mu)=0$. This is our desired contradiction since it prevents $\vy \in \Int(\J_\p(\D))$ by \Cref{coro.criterionboundary}. The claim then follows by the sequence of identities 
    \begin{align*}
        \frac{\ell_{[\hat\vv]}(\mu)}{\ell_\p(\mu)}  = \langle \bfL_\p(\mu),\hat\vv\rangle +\Dil(\D_{-\hat\vv},\psi)
         = \lim_{m \to \infty}{\frac{\langle -\nabla\theta(\vv_m),\vv_m\rangle +\thet(\vv_m)}{\|\vv_m\|}}=0,
    \end{align*} 
    which hold by \Cref{lem.limitthetaboundary}, \Cref{lem.ineqconvexity}, and the divergence $\|\vv_m\| \to \infty$. This completes the proof of the proposition.
\end{proof}

We conclude with the proof of \Cref{prop.descriptionboundary}.

\begin{proof}[Proof of \Cref{prop.descriptionboundary}]
  We first prove Item (1), so suppose that $\partial \J_\p(\D)$ is strictly convex. This is equivalent to $N_\J(\vx) \cap N_\J(\vy)=\emptyset$ for all $\vx \neq \vy$ in $\partial \J_\p(\D)$. That means, for any $[\vv] \in \mathbb{S}^{n-1}$ non-zero there exists a unique $\Phi([\vv]):=\vx$ in $\partial \J_\p(\D)$ such that $[\vv]\in N_\J(\vx)$. Then we use $\Phi$ to extend $-\nabla \theta:\R^n \ra \Int(\J_\p(\D))$ to a function $\Phi:\ov{\R^n}\ra \J_\p(\D)$, which is clearly surjective. To prove continuity, by \Cref{thm.homeointeriors} it is enough to check continuity along a divergent sequence $\vv_m$ in $\R^n$ such that $\hat\vv_m \to \hat \vv$ in $\mathbb{S}^{n-1}$. Since $\J_\p(\D)$ is compact, suppose after extracting a subsequence that $-\nabla \thet(\vv_m) \to \vx$. Then we must have $\vx \in \partial \J_\p(\D)$ by \Cref{prop.IntJ}, and by \Cref{lem.limitthetaboundary} we get
  \begin{equation}\label{eq.identitydilinnerproduct}
      -\langle \hat \vv, \vx\rangle =\lim_{m \to \infty} \frac{\langle \vv_m, \thet(\vv_m)\rangle}{\|\vv_m\|}=\Dil(\D_{-\hat \vv},\p).
  \end{equation}
  If $\vx=\bfL_\p(\mu)$ for $\mu\in \calC_\G$ given by \Cref{lem.currentsJTS}, then \eqref{eq.identitydilinnerproduct} gives us $\ell_{\p_{[\hat \vv]}}(\mu)=0$ and hence $[\vv]\in N_\J(\vx)$ by \Cref{lem.criterionnormal}. This forces $\vx=\Phi([\hat\vv])$, which completes the proof of continuity.

  Now we deal with the proof of Item (2), so suppose that $\partial \J_\p(\D)$ is $C^1$. This is equivalent to $\# N_\J(\vx)=1$ for each $\vx \in \partial \J_\p(\D)$. We let $\Psi(\vx)\in \mathbb{S}^{n-1}$ be such that $N_\J(\vx)=\{\Psi(\vx)\}$. This gives us an extension $\Psi: \J_\p(\D) \ra \ov{\R^n}$ of $(-\nabla \thet)^{-1}:\Int(\J_\p(\D)) \ra \R^n$, which is clearly surjective. To prove continuity, it is enough to check convergence of $\Psi$ along a sequence $\vx_m$ in $\Int(\J_\p(\D))$ converging to $\vx \in \partial \J_\p(\D)$. Suppose that $\vx_m=-\nabla \thet(\vv_m)$ for each $m$, so that $\vv_m$ diverges in $\R^n$ by \Cref{prop.IntJ}. After extracting a subsequence we can assume that $\hat\vv_m \to \hat \vv$ in $\mathbb{S}^{n-1}$.
  As in the proof of \eqref{eq.identitydilinnerproduct}, we can prove that $-\langle \hat \vv,\vx\rangle =\Dil(\D_{-\hat\vv},\p)$ and that $\ell_{\p_{[\hat \vv]}}(\mu)=0$ for any $\mu\in \calC_\G$ such that $\bfL_\p(\mu)=\vx$. Then $[\hat \vv]\in N_\J(\vx)$ by \Cref{lem.criterionnormal} and hence $[\hat \vv]=\Psi(\vx)$.
\end{proof}

\section{Examples}\label{sec.examples}
In this section we discuss examples of the joint translation spectrum in more detail.

\subsection{Word metrics}
Suppose that $S, S_1,S_2, \ldots, S_n$ are finite symmetric generating sets for $\G$ and write $d,d_1,\ldots,d_n$ for the corresponding word metrics in $\Dc_\G$. In this section we study the joint spectrum $\J_d(d_1,\ldots,d_n).$ We let $\D=(d_1,\dots,d_n)$.

\begin{proposition}\label{prop.wordpoly}
    Suppose that $d, d_1, \ldots, d_n$ are all word metrics,  then $\J_{d}(d_1,\ldots,d_n)$ is a polytope.
\end{proposition}
\begin{proof}
The proof relies on realising the joint translation spectrum as the image of the so-called \textit{weight per symbol} function for a Markov chain. To do this we need to use the Cannon coding which we discussed in the proof of \Cref{lem.construct}. To avoid introducing a lot of preliminary material for the Cannon coding here we refer the reader to Section 3 of \cite{cantrell.mixing}.

The key result that allows us to represent the joint translation spectrum as an object related to a Markov chain is due to Calegari and Fujiwara \cite{calegari-fujiwara}. They show that, given a hyperbolic group $\G$ and generating sets $S,S_\ast$ there is a Cannon coding $(\mathcal{G}, V,E, v_\ast)$ (where we use the notation introduced in the proof of \Cref{lem.construct}) for $(\G,S)$ and a function $\phi_{S_\ast} : E \to \Z$ that encodes $S_\ast$ in the following way: for any finite path starting at $v_\ast$ following the edges $e_1,e_2, \ldots, e_k$, the group element $x \in \G$ corresponding to this path has $S_\ast$ word length
\[
|x|_{S_\ast} = \sum_{j=1}^k \phi_{S_\ast}(e_j),
\]
i.e. there is a labeling of the edges of a Cannon graph for $S$ that encodes the $S_\ast$ word distance. In fact, by \cite[Lemma~6.42]{calegari.scl} we can find a Cannon graph $\mathcal{G}$ for $S$ and a collection of functions $\phi_1, \ldots, \phi_n$ that represent the corresponding generating sets $S_1,\ldots,S_n$ simultaneously. We can then define $\Phi: E \to \Z^n$ by
\[
\Phi(e) = (\phi_1(e), \ldots, \phi_n(e)) \ \ \text{ for each $e \in E$}
\]
and extend $\Phi$ to finite paths in the coding by summing along the edges in the path. The function sending a finite path $p$ to
\[
 p \mapsto \frac{\Phi(p)}{|p|} \ \ \text{ where $|p|$ is the number of edges in $p$}
\]
is called the \emph{weight-per-symbol function} \cite{marcus-tuncel}.
Further it is easy to check that if $p$ is a loop in the Cannon graph with $|p| = k$ the $p$ corresponds to a group element $x_p$ with $\ell_S(x_p) = |x_p|_S = k$ and furthermore $\Phi(x_p) = \D(x_p) = \Lambda(x_p).$

We now recall that by the work of Marcus and Tuncel \cite{marcus-tuncel} the image of the weight per symbol function over the closed paths in a mixing Markov chain is a polytope. It is easy to see that this is also the case when the Markov chain is transitive instead of mixing. The Cannon graph may have multiple connected components that give rise to transitive Markov chains. Further by Corollary 3.2 in \cite{cantrell.mixing} there exists a connected component $\Cc$ that sees all of the conjugacy classes in $\G$ in the sense that for any non-torsion $x \in \G $ there exists a closed loop $p$ inside $\Cc$ such that the group element $x_p$ corresponding to $p$ belongs to $x^M$ for some integer $M \ge 1$. Hence we deduce that
\[
\left\{\frac{\Phi(x_p)}{|p|} : \text{ $p$ is loop in $\Cc$} \right\} = \{ \Lambda_d(x) : \text{ $x \in \G$ is a non-torsion}\}.
\]
It then follows from \Cref{thm.intro.jtspectrum} that the joint spectrum is exactly the image of a weight per symbol function for a transitive Markov chain, which, by the above discussion, is a polytope.
\end{proof}

In this case we also have the following upgrade of \Cref{thm.manc^1}.

\begin{corollary}\label{coro.analyticword}
    When $d,d_1, \ldots, d_n$ are all word metrics, the corresponding Manhattan manifold is analytic.
\end{corollary}

\begin{proof}
In the case of Manhattan curves (i.e. when $n=1$) his was proven by Cantrell and Tanaka \cite[Theorem 5.5]{cantrell-tanaka.invariant}. The same proof can be applied for $n > 1$ with minor changes: we leave the details to the reader.
\end{proof}

\subsection{Cubulations}
A \emph{$\CAT(0)$ cube complex} is a simply-connected,  metric polyhedral complex in which all polyhedra are unit length Euclidean cubes and satisfying a form of non-positive curvature \cite{bridson-haefliger,sageev.cat0}. If $\calX$ is a $\CAT(0)$ cube complex, we are interested in its \emph{combinatorial metric}, which can be defined as the graph metric on its 1-skeleton so that each edge has length 1. If $\G$ acts cubically on $\calX$, then this action preserves the combinatorial metric, and we can pullback this metric along any $\G$-orbit to get a left-invariant pseudo metric $d_\calX$ on $\G$.

Suppose now that $\G$ is non-elementary hyperbolic and acts geometrically and cubically on the $\CAT(0)$ cube complexes $\calX,\calX_1,\dots,\calX_n$. As for word metrics, we can prove the following.


\begin{proposition}\label{prop.polytopecubulations}
     Let $d, d_1, \ldots, d_n\in \calD_\G$ be pseudo metrics on $\G$ induced by the combinatorial metrics associated to the geometric actions on the $\CAT(0)$ cube complexes $\calX,\calX_1,\dots,\calX_n$ as above. Then for $\J_{d}(d_1,\ldots,d_n)$ is a polytope and the corresponding Manhattan manifold is analytic.
\end{proposition}
\begin{proof}
We follow the same proof as that presented in \Cref{prop.wordpoly}. To do so we need analogues of the results of Calegari and Fujiwara for word metrics on hyperbolic groups but for our actions on $\text{CAT}(0)$ cube complexes. These results were obtained by Cantrell and Reyes in \cite{cantrell-reyes.statistic}. More precisely they show that there exists an automatic structure (a finite directed graph) that represents the action of a finite index subgroup $\G'<\G$ on $\calX$ (i.e. paths in the graph correspond to combinatorial geodesics for the action). Furthermore we can find such a structure for which there are labelings of the edges with lengths that represent the action on $\calX_1, \ldots, \calX_n$ \cite[Lemma 6.6]{cantrell-reyes.statistic}. That is. 
the sum of the labelings along a path represents the combinatorial distance for the actions of $\G'$ on $\calX_1, \ldots, \calX_n$ corresponding to the group element represented by the path (Lemma 6.6 in \cite{cantrell-reyes.statistic} is only stated for a pair of actions, however the same argument allows us to encode arbitrarily finitely many actions). It follows that, as for word metrics on $\G$, $\J_{d}(d_1,\ldots,d_n)$ is the image of a weight-per-symbol function for a transitive Markov chain (the fact we can take a single component follows from \cite[Lemma 6.13]{cantrell-reyes.statistic}) and so is a polytope by the work of Marcus and Tuncel \cite{marcus-tuncel}.

Analiticity of the associated Manhattan manifold follows exactly as the proof of \Cref{coro.analyticword} (cf.~\cite[Theorem~6.1]{cantrell-reyes.statistic}).
\end{proof}

\subsection{Cells in Outer space}
Let $\G=F_k$ be a free group on $k\geq 2$ generators. The \emph{Culler-Vogtmann outer space} of $\G$ \cite{culler-vogtmann} is the space $\scrC\scrV_k$ of rough homothety classes of geometric and minimal actions of $F_k$ on metric trees. By looking at marked homeomorphism types of quotients for these actions, we get a decomposition of $\scrC\scrV_k$ as a union of open cells. We can describe these cells in $\scrD_\G$ using Manhattan manifolds. 

As a concrete example, suppose $k=3$ and let $a,b,c$ be a basis for $F_3$. Let $\mathsf{C}\subset \scrC\scrV_k$ be the cell corresponding by $F_3$ actions on trees $F_3$-equivariantly BiLipschitz to the action of $F_3$ on its Cayley graph with generators $S=\{a^{\pm},b^{\pm},c^{\pm}\}$. Equivalently, points in $\mathsf{C}$ are metric structures induced by the metrics 
\[d_{v_a,v_b,v_c}(x,y)=v_a d_a(x^{-1}y)+v_b d_b(x^{-1}y)+v_c d_c(x^{-1}y), \]
where $d_a,d_b,d_c$ count the number of appearances of $a^\pm,b^\pm$ and $c^\pm$ respectively in irreducible representatives and
$v_a,v_b,v_c>0$. 

If $d=d_{1,1,1}$ be the word metric for $S$ and $\D=(d_a,d_b,d_c)$, then it is not hard to check that the Manhattan manifold $\scrM_{\D/d}$ is exactly $\mathsf{C}$. Note in this case that the boundary metric structures in $\partial \scrM_{\D/d}$ are of the form $d_{v_a,v_b,v_c}$, where $v_a,v_b,v_c\geq 0$, at least one of these numbers is positive and at least one of them is zero. These boundary points correspond to small (but not proper) actions of $\G$ on trees. Also, the joint translation spectrum is the set of all triples $(x_a,x_b,x_c)$ in $\R^3$ with $x_a+x_b+x_c=1$ and $x_a,x_b,x_c \geq 0$, hence a polytope.

The same phenomenon applies to the other cells in outer space.

\subsection{Stable norm}

Suppose now that $\G$ has positive first Betti number and consider the natural map $\pi: \G \ra H_1(\G)=H_1(\G;\R)$. For $d\in \calD_\G$ we first define a left-invariant pseudo metric $\ov d$ on $H_1(\G;\Z)$ according to
\[\ov{d}(\ov x, \ov y)=\inf\{d(x,y) : x\in \pi^{-1}(\ov {x}), y\in \pi^{-1}(\ov {y})\}.\]
This pseudo metric induces a semi-norm $\|\cdot \|_d$ on $H_1(\G)$ given by the extension of
\[
\|\ov{x}\|_d=\lim_{k \to \infty}{\frac{\ov d( \ov{o},\ov{x}^k)}{k}}
\]
for $x\in H_1(\G;\Z)$. This semi-norm is indeed a norm and depends only on the rough isometry class of $d$. We call $\|\cdot \|_d$ the \emph{stable norm} associated to $d$. 

Now we let $a_1,\dots,a_n\in \G$ be such that $\pi(a_1),\dots,\pi(a_n)$ form a basis of $H_1(\G)$ with dual basis $\p_1,\dots,\p_n \in H^1(\G)\subset \calH_\G$. The basis $\pi(a_1),\dots,\pi(a_n)$ then gives a linear identification $\omega:H_1(\G) \ra \R^n$. For this identification, we have the following:

\begin{proposition}\label{prop.stablenorm}
The unit ball of $\|\cdot \|_d$ in $H_1(\G)=\R^n$ is exactly the joint translation spectrum $\J_d(\D)$. 
\end{proposition}

\begin{proof}
Let $\calB$ denote the unit ball of $\|\cdot\|_d$ in $H_1(\G)$. 
From the definition of the stable norm it is clear that $\|\pi(x)\|_d\leq \ell_d(x)$ for all $x\in \G$. In particular, if $x$ is non-torsion then $\frac{1}{\ell_d(x)}\pi(x)$ belongs to $\calB$. 
But we also have
\[\Lam_d(x)=\frac{1}{\ell_d(x)}(\p_1(x),\dots,\p_n(x))=\omega^{-1}\left(\frac{1}{\ell_d(o,x)}\pi(x)\right),\]
an hence $\om(\Lam_d(x))$ belongs to $\calB$ and $\om(\J_d(\D))\subset \calB$.

To prove the reverse inclusion, let $\vx\in H_1(\G)$ and write $\vx=\lim_{m\to \infty}{\vx_m}$ for $\vx_m=\lam_m\pi(x_m)$ with $\lam_m>0$ and $x_m\in \G$. By appropriately choosing the $x_m$'s we can assume that each $x_m$ is non-torsion and that  
\[
(1-1/m)\lam_m \ell_d(x_m)  \leq \|\vx_m\|_d \leq (1+1/m)\lam_m \ell_d(x_m).    \]

From this we obtain
\[\omega(\Lam_d(x_m))=\frac{1}{\ell_d(x_m)}\pi(x_m)=\frac{1}{\lam_m \ell_d(x_m)}\vx_m \to \frac{1}{\|\vx\|_d}\vx.\]

If $\|\vx\|_d=1$, then $\vx$ belongs to $\omega(\J_d(\D))$ since $\J_d(\D)$ is compact. In general, if $\vx \in \calB$ is non-zero then it is a convex combination of the trivial element of $H_1(\G)$ and $\frac{1}{\|\vx\|_d}\vx$, and hence $\vx$ belongs to $\omega(\J_d(\D))$ by \Cref{thm.intro.jtspectrum}.
\end{proof}

A classical example of stable norm is when $\G$ is a closed hyperbolic surface group and $d\in \calD_\G$ is induced by a Fuchsian representation \cite{pollicott-sharp.livsic,massart}. Indeed, the work of Massart \cite{massart} implies that the unit ball has a corner at infinitely many directions in $H_1(\G;\Z)$. Combining this with \Cref{prop.stablenorm} provides an explicit example of a joint translation spectrum $\J_d(\D)$ that is not a polyhedron.

\section{Further directions}\label{sec.questions}

We end with some questions related to the joint translation spectrum. Our first question is about the speed of convergence of the sets involved in the definition of the joint translation spectrum.

\begin{question}
    How quickly do $\D_\p(S_R(T))$ or $\Lam(S_R^{\ell_\p}(T))$ converge to $\J_\p(\D)$ as $T$ increases? Is there a version of Breuillard-Fujiwara's inequality \cite[Theorem~1.4]{breuillard-fujiwara} in this setting? 
\end{question}

We also ask about realizability of convex bodies as joint translation spectra.

\begin{question}
    For fixed $\G$, what convex sets can be realized as the joint translation spectrum for some $\p_1,\dots,\p_n \in \calH_\G$ and $\p\in \calH_\G^{++}$? What convex sets can be realized as the joint translation spectrum for $\p_1,\dots,\p_n,\p\in \calD_\G$?
\end{question}

Extending the random walk spectrum, we let $\mathcal{WJ}^1_\p(\D)$ denote set of typical drifts for $\D_\p$ for random walk as in \Cref{def.rw}, except that we allow finite first moment random walks (instead of only finitely supported). In virtue of \Cref{thm.wj=j}, we ask:

\begin{question}
    Do we always have $\mathcal{WJ}^1_{\p}(\D) = \J_{\p}(\D)$? 
\end{question}

Similarly, from \Cref{thm.dj=j} we know that $\mathcal{DJ}_d(\psi_1,\ldots,\psi_n)$ contains the interior of $\mathcal{J}_d(\psi_1,\ldots,\psi_n)$ under the assumption of independence. The behavior of intersection $\partial \J_\p(\D) \cap \calD\J_\p(\D)$ remains unclear.

\begin{question}
Can $\calD\J_\p(\D)$ contain at least an element of the boundary of $\J_\p(\D)$? If so, can it contain a point at the boundary but not the whole boundary?
\end{question}

From a dynamical perspective, it is expected that the joint translation spectrum is in general strictly convex and not $C^1$.

\begin{question}
If there any interesting class of hyperbolic metric potentials for which for (independent) tuples among them the joint translation spectrum is strictly convex? Is this the case for the class of all Green metrics?
\end{question}

In general, is desirable to have a complete characterization of regularity and convexity of the boundary $\partial \J_\p(\D)$. More precisely, we ask:

\begin{question}
    Do the converses of \Cref{prop.descriptionboundary} hold?
\end{question}

\subsection*{Open access statement}
For the purpose of open access, the authors have applied a Creative Commons Attribution (CC BY) licence to any Author Accepted Manuscript version arising from this submission.

\noindent\small{Department of Mathematics, 
University of Warwick,
Coventry, CV4 7AL, UK}\\
\small{\textit{Email address}: \texttt{stephen.cantrell@warwick.ac.uk, cagri.sert@warwick.ac.uk}\\
\\
\small{Department of Mathematics, Yale University, New Haven, CT 06511, USA}\\
\small{\textit{Email address}: \texttt{eduardo.c.reyes@yale.edu}}\\


\begin{thebibliography}{100}

\bibitem{AMS} 
H.~Abels, G.~Margulis and G.~Soifer, Semigroups containing proximal linear maps. \textit{Israel J. Math.} \textbf{9} (1995), 1--30.

\bibitem{ambrosio.book}
L.~Ambrosio, P.~Tilli. Topics on analysis in metric spaces. Oxford University Press, \textbf{25}, Oxford, 2004.

\bibitem{barabanov}
N.~E.~Barabanov, On the Lyapunov exponent of discrete inclusions. \textit{I. Automat. Remote Control} \textbf{49} (1988), 152--157.

\bibitem{benoist.cone}
Y.~Benoist, Propri\'et\'es asymptotiques des groupes lin\'eaires. \textit{Geom. Funct. Anal.} \textbf{7} (1997), 1--47.

\bibitem{berger-wang}
M.~Berger, Y.~Wang, Bounded semigroups of matrices. \textit{Linear Algebra Appl.} \textbf{166} (1992), 21--27.

\bibitem{BHM-greenspeed}
S.~Blach\'ere, P.~Ha\"issinsky and P.~Mathieu, Asymptotic entropy and Green speed for random walks on countable groups. \textit{Ann. Probab.} \textbf{36} (2008), no. 3, 1134--1152.

\bibitem{BHM.harmonic}
S.~Blach\'ere, P.~Ha\"issinsky and P.~Mathieu, Harmonic measures versus quasiconformal measures for hyperbolic groups. \textit{Ann. Sci. \'Ec. Norm. Sup\'er.} (4) \textbf{44} (2011), no. 4, 683--721.(hal-00290127v2)

\bibitem{bochi} 
J.~Bochi, Inequalities for numerical invariants of sets of matrices. \textit{Linear Algebra Appl.} \textbf{368} (2003), 71--81.

\bibitem{BPS} 
J.~Bochi, R.~Potrie and A.~Sambarino, Anosov representations and dominated splittings. \textit{J. Eur. Math. Soc.} \textbf{21} (2019), 3343--3414.

\bibitem{bonahon.currentshypgroups}
F.~Bonahon, Geodesic currents on negatively curved groups. \textit{Arboreal group theory} (Berkeley, CA, 1988), 143–168, Math. Sci. Res. Inst. Publ., \textbf{19}, Springer, New York, 1991.

\bibitem{bonk-schramm}
M.~Bonk, O.~Schramm, Embeddings of Gromov hyperbolic spaces. \textit{Geom. Funct. Anal.} \textbf{10} (2000), 266--306.

\bibitem{BMSS}
A.~Boulanger, P.~Mathieu, C.~Sert and A.~Sisto, Large deviations for random walks on Gromov-hyperbolic spaces. \textit{Ann. Scient. Ec. Norm. Sup.} \textbf{56} (2023), 885--944.

\bibitem{bousch-mairesse}
T.~Bousch, J.~Mairesse, Asymptotic height optimization for topical ifs, tetris heaps, and the finiteness conjecture. \textit{J. Amer. Math. Soc.} \textbf{15} (2002), 77--111.

\bibitem{breuillard-fujiwara}
E.~Breuillard, K.~Fujiwara, On the joint spectral radius for isometries of non-positively curved spaces and uniform growth. \textit{Ann. Inst. Fourier} \textbf{71} (2021), no. 1, 317--391.

\bibitem{breuillard-sert}
E.~Breuillard, C.~Sert, The joint spectrum. \textit{J. Lond. Math. Soc. (2)} \textbf{103} (2021), 943--990.

\bibitem{bridson-haefliger}
M.~Bridson, A.~Haefliger, Metric spaces of non-positive curvature. Grundlehren der Mathematischen Wissenschaften, \textbf{319}, Springer-Verlag, Berlin, 1999.

\bibitem{brouwer}
L.~E.~J.~Brouwer, Zur Invarianz des $n$-dimensionalen Gebiets. \textit{Math. Ann.} \textbf{72} (1912), 55--56.

\bibitem{burger} 
M.~Burger, Intersection, the Manhattan curve, and Patterson-Sullivan theory in rank 2. \textit{Inter. Math. Res. Not.} \textbf{7} (1993), 217--225.

\bibitem{calegari.scl}
D.~Calegari, scl. MSJ Memoirs, \textbf{20}, Tokyo, 2009.

\bibitem{calegari-fujiwara}
D.~Calegari, K.~Fujiwara, Combable functions, quasimorphisms, and the central limit theorem. \textit{Ergodic Theory Dynam. Systems} \textbf{30} (2010), no. 5, 1343--1369.

\bibitem{cannon}
J.~W.~Cannon, The combinatorial structure of cocompact discrete hyperbolic groups. \textit{Geom. Dedicata}  \textbf{16} (1984), no. 2, 123--148.


\bibitem{cantrell.mixing} S.~Cantrell. Mixing of the Mineyev flow, orbital counting and Poincar\'e series for strongly hyperbolic metrics.
\url{https://arxiv.org/abs/2210.11558}, arXiv preprint, 2022.


\bibitem{CMGR}
S.~Cantrell, D.~Mart\'inez-Granado, E.~Reyes.
Density of Green metrics for hyperbolic groups. In preparation.

\bibitem{cantrell-reyes.approx}
S.~Cantrell, E.~Reyes. Approximate marked length spectrum rigidity in coarse geometry. 
\url{https://arxiv.org/abs/2410.01965}, arXiv preprint, 2024.

\bibitem{cantrell-reyes.manhattan} 
S.~Cantrell, E.~Reyes, Manhattan geodesics and the boundary of the space of metric structures on hyperbolic groups. \textit{Comment. Math. Helv.} (2024). Published online first.

\bibitem{cantrell-reyes.MLS} S.~Cantrell, E.~Reyes. Marked length spectrum rigidity from rigidity on subsets.
\url{https://arxiv.org/abs/2304.13209}, arXiv preprint, 2023.

\bibitem{cantrell-reyes.statistic}
S.~Cantrell, E.~Reyes.
Rigidity phenomena and the statistical properties of group actions on $\mathrm{CAT}(0)$ cube complexes.
\url{https://arxiv.org/abs/2310.10595}, arXiv preprint, 2023.

\bibitem{cantrell-tanaka.invariant}
S.~Cantrell, R.~Tanaka, Invariant measures of the topological flow and measures at infinity on hyperbolic groups. \textit{J. Mod. Dyn.}
\textbf{20} (2024), 215--274.

\bibitem{cantrell-tanaka.manhattan}
S.~Cantrell, R.~Tanaka, The Manhattan curve, ergodic theory of topological flows and rigidity, to appear in Geometry \& Topology, arXiv:2104.13451, 2021.

\bibitem{choi.random1}
I.~Choi. Random walks and contracting elements I: Deviation inequality and limit laws. \url{https://arxiv.org/abs/2207.06597}, arXiv preprint, 2022.

\bibitem{coornaert}
M.~Coornaert, Mesures de Patterson-Sullivan sur le bord d'un espace hyperbolique au sens de Gromov. \textit{Pacific J. Math.} \textbf{159} (1993), 241--270.

\bibitem{culler-vogtmann}
M.~Culler, K.~Vogtmann, Moduli of graphs and automorphisms of free groups. \textit{Invent. Math.} \textbf{84} (1986), no. 1, 91--119.

\bibitem{daubechies-lagarias}
I.~Daubechies, J.~C.~Lagarias, Sets of matrices all infinite products of which converge, \textit{Linear Algebra Appl.} \textbf{162} (1992), 227--261.

\bibitem{delzant-steenbock}
T.~Delzant, M. Steenbock, Product set growth in groups and hyperbolic
geometry. \textit{J. Topol.} \textbf{13} (2020), 1183--1215.

\bibitem{epstein-fujiwara}
D.~B.~A.~Epstein, K.~Fujiwara, The second bounded cohomology of word-
hyperbolic groups. \textit{Topology} \textbf{36} (1997), 1275--1289.

\bibitem{furman}
A.~Furman, Coarse-geometric perspective on negatively curved manifolds and groups. Rigidity in dynamics and geometry (Cambridge, 2000), pp. 149–166. Springer, Berlin, 2002.

\bibitem{ghys-delaharpe}
\'E.~Ghys, P.~de~la~Harpe, Sur les groupes hyperboliques d'apr\`es Mikhael Gromov. Progress in Mathematics, \textbf{83}. Birkhäuser Boston, Inc., Boston, MA, 1990.

\bibitem{gouezel.schottky}
S.~Gou\"ezel, Exponential bounds for random walks on hyperbolic spaces without moment conditions. \textit{Tunisian Journal of Mathematics} \textbf{4} (2023), 635--671.

\bibitem{guivarch-raugi}
Y.~Guivarc'h, A.~Raugi, Frontiere de Furstenberg, propri\'et\'es de contraction et th\'eoremes de convergence. \textit{Zeitschrift f\"ur Wahrscheinlichkeitstheorie und Verwandte Gebiete} \textbf{69} (1985), 187--242.

\bibitem{gurvits}
L.~Gurvits, Stability of discrete linear inclusion. \textit{Linear Algebra Appl.} \textbf{231} (1995), 47--85.

\bibitem{jenkinson.survey}
O.~Jenkinson, Ergodic optimization in dynamical systems. \textit{Ergodic Theory Dynam. Systems} \textbf{39} (2019), 2593--2618.

\bibitem{kapovich.martinezgranado}
M.~Kapovich, D.~Mart\'nez-Granado. Bounded backtracking property and geodesic currents. In preparation.

\bibitem{lagarias-wang}
J.~C.~Lagarias, Y.~Wang, The finiteness conjecture for the generalized spectral radius of a set of matrices. \textit{Linear Algebra Appl.} \textbf{214} (1995), 17--42.

\bibitem{marcus-tuncel} 
B.~Marcus, S.~Tuncel, Entropy at a weight-per-symbol and embeddings of Markov chains. \textit{Invent. Math.} \textbf{102} (1990), 235--266.

\bibitem{massart}
D.~Massart, Stable norms of surfaces: local structure of the unit ball of rational directions. \textit{Geom. Funct. Anal.} \textbf{7} (1997), 996--1010.

\bibitem{mineyev.flow}
I.~Mineyev, Flows and joins of metric spaces. \textit{Geom. Topol.} \textbf{9} (2005), 403--482.

\bibitem{morris.mather}
I.~D.~Morris, Mather sets for sequences of matrices and applications to the study of joint spectral radii. \textit{Proc. Lond. Math. Soc. (3)} \textbf{107} (2013), 121--150.

\bibitem{nica-spakula}
B.~Nica, J.~\v{S}pakula, Strong hyperbolicity. \textit{Groups Geom. Dyn.} \textbf{10} (2016), no. 3, 951--964.

\bibitem{pollicott-sharp.livsic}
M.~Pollicott, R.~Sharp, Livsic theorems, maximizing measures and the stable norm. \textit{Dyn. Syst.} \textbf{19} (2004), 75--88.

\bibitem{oregonreyes.inequalities}
E.~Reyes, Properties of sets of isometries of Gromov hyperbolic spaces. \textit{Groups Geom. Dyn.} \textbf{12} (2018), no. 3, 889--910.

\bibitem{oregonreyes.metric}
E.~Reyes, The space of metric structures on hyperbolic groups. \textit{J. Lon. Math. Soc.} \textbf{107} (2023), 914--942.

\bibitem{rota-strang}
G.~C.~Rota, G.~Strang, A note on the joint spectral radius. \textit{Indag. Math.} \textbf{22} (1960), 379--381.

\bibitem{sageev.cat0}
M.~Sageev, CAT(0) cube complexes and groups. In Geometric group theory, volume 21 of \textit{IAS/Park City Math. Ser.}, pages 7--54. Amer. Math. Soc., Providence, RI, 2014.

\bibitem{tanaka.topflows}
R.~Tanaka,Topological flows for hyperbolic groups. \textit{Ergodic Theory Dynam. Systems} \textbf{41} (2021), 3474--3520.




\end{thebibliography}
\end{document}